\documentclass[reqno,11pt]{amsart}

\usepackage{amssymb} 

\usepackage{amsmath}

\usepackage{amsthm} 
\usepackage{ascmac} 
\usepackage{bm} 
\usepackage{color} 
\usepackage[OT2,T1]{fontenc} 
\usepackage{fullpage} 

\usepackage{indentfirst} 
\usepackage{stmaryrd}
\usepackage{mathrsfs} 
\usepackage{multirow} 
\usepackage{stmaryrd} 
\usepackage[all]{xy}
\usepackage{graphicx} 
\usepackage{tikz-cd}
\usepackage{enumitem}
\usetikzlibrary{decorations.markings}
\newcommand*{\halfway}{0.5*\pgfdecoratedpathlength+.5*3pt}
\usetikzlibrary{shapes.geometric}
\tikzset{every loop/.style={min distance=10mm,looseness=10}}
\usepackage{mathtools} 
\allowdisplaybreaks[4]
\mathtoolsset{showonlyrefs=true}

\theoremstyle{plain}
\newtheorem{thm}{Theorem} 
\newtheorem{prop}[thm]{Proposition}
\newtheorem{lem}[thm]{Lemma}

\numberwithin{thm}{section} 
\theoremstyle{definition}
\newtheorem{defi}[thm]{Definition}
\newtheorem{conj}{Conjecture}
\newtheorem{ex}[equation]{Example}
\newtheorem{rem}[thm]{Remark}

\newtheorem{assump}{Assumption}

\let\Im\relax
\DeclareMathOperator{\Im}{Im}
\let\epsilon\relax
\DeclareMathOperator{\epsilon}{\varepsilon}
\makeatletter
\newcommand{\MyMathOperators}[1]{\@for\op:=#1\do{
	\expandafter\edef\csname\op\endcsname{\noexpand\mathop{\noexpand\mathrm{\op}}\nolimits}
}}
\makeatother
\MyMathOperators{%
	Ker,%
	Coker,%
	Hom,%
	End,%
	Aut,%
	id,%
	Gal,%
	ch,%
	ord,%
	Spec,%
	Specm,%
	Frac,%
	rank,%
	corank,%
	sgn,%
	alg,%
	Reg,%
	Sel,%
	Cl,%
	lcm,%
	gcd,%
	Ann,%
	alg,%
}

\renewcommand{\bar}[1]{\overline{#1}}
\renewcommand{\tilde}[1]{\widetilde{#1}}
\renewcommand{\hat}[1]{\widehat{#1}}

\DeclareSymbolFont{cyrletters}{OT2}{wncyr}{m}{n}
\DeclareMathSymbol{\Sha}{\mathalpha}{cyrletters}{"58}
\DeclareMathOperator{\inc}{inc}

\DeclareMathOperator{\gal}{Gal}
\DeclareMathOperator{\indu}{Ind}
\DeclareMathOperator{\cl}{Cl}
\DeclareMathOperator{\auto}{Aut}

\DeclareMathOperator{\tra}{Tr}
\DeclareMathOperator{\autom}{Aut}
\DeclareMathOperator{\spec}{Spec}

\newcommand{\Z}{\mathbb{Z}}
\newcommand{\Q}{\mathbb{Q}}

\newcommand{\F}{\mathbb{F}}
\newcommand{\E}{\mathbb{E}}
\usepackage{hyperref}
\hypersetup{pdffitwindow=true,linkcolor=calpolypomonagreen,citecolor=calpolypomonagreen,urlcolor=calpolypomonagreen}

\title{Iwasawa theory for weighted graphs}
\author{Taiga Adachi}
\address{Taiga Adachi \newline JOINT GRADUATE SCHOOL OF MATHEMATICS FOR INNOVATION, KYUSHU UNIVERSITY, MOTOOKA 744, NISHI-KU FUKUOKA 819-0395, JAPAN}
\email{t.adachi1729@gmail.com}
\author{Kosuke Mizuno}
\address{Kosuke Mizuno \newline GRADUATE SCHOOL OF MATHEMATICS, NAGOYA UNIVERSITY, FURO-CHO, CHIKUSA-KU, NAGOYA, 464-8602, JAPAN}
\email{kosuke.mizuno.c1@math.nagoya-u.ac.jp}
\author{Sohei Tateno}
\address{Sohei Tateno \newline GRADUATE SCHOOL OF MATHEMATICS, NAGOYA UNIVERSITY, FURO-CHO, CHIKUSA-KU, NAGOYA, 464-8602, JAPAN}
\email{inu.kaimashita@gmail.com}
\date{}

\begin{document}

\begin{abstract}
   Let $p$ be a prime number and let $d$ be a positive integer. In this paper, we generalize Iwasawa theory for graphs initiated by Gonet and Valli\`{e}res to weighted graphs. In particular, we prove an analogue of Iwasawa's class number formula and that of Kida's formula for compatible systems of $(\Z/p^n\Z)^d$-covers of weighted graphs. We also provide numerical examples of characteristic elements and Iwasawa invariants. At the end of this paper, we give an application of the ideas of Iwasawa theory to the theory of discrete-time quantum walks in graphs.
\end{abstract}

\maketitle
\tableofcontents
\section{Introduction}
\noindent Let $p$ be a fixed prime number. We denote the completion of $\Q$ equipped with the $p$-adic absolute value by $\Q_p$ and its valuation ring by $\Z_p$. Fix an algebraic closure $\bar{\Q}_p$ of $\Q_p$. Let $v_p$ be the $p$-adic valuation of $\Q_p$ normalized so that $v_p(p)=1$, and we extend this to the valuation of $\bar{\Q}_p$ uniquely. For an integer $a$, $v_p(a)$ is the number of times that $a$ is divisible by $p$. In number theory, for a positive integer $d$ and a number field $K$, celebrated Iwasawa's class number formula and its generalization by Cuoco--Monsky \cite{CM81} describe the asymptotic behavior of class numbers in a tower of subfields of a $\Z_p^d$-extension $K_{\infty}\slash K$
\begin{align}
K=K_0\subset K_1\subset K_2\subset \cdots \subset K_n\subset \cdots \subset K_{\infty}\label{seq}
\end{align}
such that $K_n\slash K$ is the subextension corresponding to the subgroup $(p^n\Z)^d$  for every $n\in\Z_{\geq 0}$. Greenberg conjectured the following (See \cite[\S7]{CM81}, \cite[\S1]{DV23}).
\begin{conj}[Greenberg's class number conjecture]\label{conj:Greenberg}
For a tower \eqref{seq}, there exists a polynomial $f(U,V)\in \Q[U,V]$ with total degree $\deg f\leq d$ and degree in $V$ $\deg_Vf\leq 1$ such that, for all $n\gg0$,
\begin{align}\label{eq:cnf}
v_p(\#\cl(K_n))=f(p^n,n),
\end{align}
where $\cl(K_n)$ is the ideal class group of $K_n$. Equivalently, there exist non-negative integers $\mu,\lambda$ and $\mu_1,\ldots,\mu_{d-1},\lambda_1,\ldots,\lambda_{d-1},\nu\in\Q$ such that, for all $n\gg 0$,
\begin{align}\label{eq:the_equality}
v_p(\#\cl(K_n))=(\mu p^n+\lambda n)p^{(d-1)n}+\left(\sum_{
i=1}^{d-1}(\mu_ip^n+\lambda_in)p^{(d-i-1)n}\right)+\nu
\end{align}
holds.

\end{conj}
\noindent Conjecture \ref{conj:Greenberg} coincides with Iwasawa's class number formula in the case of $d=1$ and is a refinement of \cite[Theorem 1]{CM81} in the case of $d\geq 2$. These (unique) $\mu,\lambda\in\Z_{\geq0}$ (and $\mu_1,\cdots,\mu_d,\lambda_1,\cdots,\lambda_{d-1}$) are called (generalized) Iwasawa invariants. The analogue of Conjecture \ref{conj:Greenberg} for $\Z_p^d$-towers of function fields was proved by Wang \cite{Wan19} and that for $\Z_p^d$-towers branched along a link in a homology $3$-sphere was proved by Ueki and the third author \cite{TU24}.

 Recently, Gonet and Valli\`{e}res initiated Iwasawa theory of graphs independently in \cite{Gon22} and \cite{Val21}. DuBose--Valli\`{e}res \cite{DV23} and Kleine--M\"{u}ller \cite{KM24} independently generalized it to $\Z_p^d$-situation. In Iwasawa theory of graphs, we consider a (unramified) $\Z_p^d$-tower over a graph $X$ as a graph-theoretic analogue of $\Z_p^d$-tower over a number field \eqref{seq}
\begin{align}
X=X_0\leftarrow X_1\leftarrow X_2\leftarrow \cdots \leftarrow X_n\leftarrow \cdots \label{graph_seq}
\end{align}
such that, for every $n\geq 0$, $X_n\slash X$ is the Galois covering with $\gal(X_n\slash X)\cong (\Z\slash p^n\Z)^d$ (We remark that, only recently, Gambheera--Valli\`{e}res initiated Iwasawa theory for $``$ramified$"$ $\Z_p$-towers in \cite{GV24}). For a graph $X$, the asymptotic behavior of the number of the spanning trees of $X_n$ in a $\Z_p^d$-tower \eqref{graph_seq} is controlled by Iwasawa invariants, which can be regarded as a graph-theoretic analogue of the class number of $K_n$ in a $\Z_p^d$-tower \eqref{seq}. In particular, DuBose--Valli\`{e}res \cite[Theorem 6.2]{DV23} and Kleine--M\"{u}ller \cite[Theorem 4.3]{KM24} showed that the graph-analogue of Conjecture \ref{conj:Greenberg} holds.

Gonet's construction in \cite{Gon22} is given in a parallel manner to the classical Iwasawa theory based on the theory of Iwasawa algebra and its modules, and the method of McGown--Valli\`{e}res (see \cite{Val21}, \cite{MV23}, and \cite{MV24}) is strongly based on the three-term determinant expressions of the Artin--Ihara $L$-functions of graphs. Since its study has been commenced, a lot of mathematicians including Dion, DuBose, Gambheera, Kataoka, Kleine, Lei, McGown, M\"{u}ller, and Ray have developed the theory, mostly following the development of classical Iwasawa theory.

On the other hand, Ishikawa, Hashimoto, Konno, Mizuno, Morita, Sato, and others have been making efforts to generalize the theory of Artin--Ihara $L$-functions of graphs, studied by Terras and Stark,  to that of weighted graphs. One of their motivation to do so is to utilize the theory of Artin--Ihara $L$-function for the study on quantum walks and random walks.

Since the proofs in the Iwasawa theory of graphs developed by Valli\`{e}res and others are profoundly based on the three-term determinant expressions of the Artin--Ihara $L$-functions, it is natural to deem that, if there exists a theory for weighted graphs including the three-term determinant formula, then we can generalize the Iwasawa theory of graphs to weighted graphs.

In this paper, by combining Iwasawa theory of graphs and the theory of Artin--Ihara $L$-function for weighted graphs, we construct the Iwasawa theory of weighted graphs. Our main theorems are the followings. First, we estimate the growth of \textit{weighted complexities}, which are generalizations of the numbers of spanning trees for weighted graphs.

\begin{thm}[Theorem \ref{thm:Iwasawa_Type}]\label{MainThm1}
Let $\{X_n\}_{n\geq 0}$ be a $\Z_p^d$-tower \eqref{graph_seq} over a weighted symmetric digraph $X$. Assume that all $X_n$ are connected. Then there exists $f(U,V)\in \Q[U,V]$ with total degree $\deg f\leq d$ and $\deg_Vf\leq 1$ such that, for all $n\gg0$,
\begin{align}
v_p(\kappa^W_n)=f(p^n,n),
\end{align}
where $\kappa_n^W$ is the weighted complexity of $X_n$. Equivalently, there exist $\mu\in \Q,\ \lambda\in \Z_{\geq 0}$ and $\mu_1,\ldots,\mu_{d-1},\lambda_1,\ldots,\lambda_{d-1},\nu\in\Q$ such that, for all $n\gg 0$,
\begin{align}\label{eq:the_equality}
v_p(\kappa_n^W)=(\mu p^n+\lambda n)p^{(d-1)n}+\left(\sum_{
i=1}^{d-1}(\mu_ip^n+\lambda_in)p^{(d-i-1)n}\right)+\nu
\end{align}
holds.
\end{thm}
We can regard the works of Gonet \cite[Theorem 1.1]{Gon22}, McGown--Valli\`{e}res \cite[Theorem 6.1]{MV24}, DuBose--Valli\`{e}res \cite[Theorem 6.2]{DV23}, and Kleine--M\"{u}ller \cite[Theorem 4.3]{KM24} as Theorem \ref{MainThm1} in the case of trivial weights. 

The first application of our main theorems is to extract information on signed graphs. By definition, the weighted complexity of a signed graph indicates
\[
(\mbox{the number of positive spanning trees})-(\mbox{the number of negative spanning trees}).
\]

Also, there is an analogue of Riemann--Hurwitz formula in Iwasawa theory of graphs \cite[Theorem 4.1]{RV22}, \cite[Theorem 1.1]{Kat24-2}. Riemann--Hurwitz formula describes the balance among genera, ramification indices, and extension degrees in a cover of Riemann surfaces. On the other hand, Kida's formula in classical Iwasawa theory describes the balance among $\lambda$-invariants, ramification indices, and extension degrees between a pair of $\Z_p$-extensions over number fields. The analogue of Kida's formula for graphs by \cite{RV22} describes the balance between $\lambda$-invariants in Theorem \ref{MainThm1} and covering degrees. Kataoka \cite[Theorem 1.1]{Kat24-2} generalized the result of Ray--Valli\`{e}res to ramified $\Z_p$-towers.

In this paper, we also prove an analogue of Kida's formula for weighted graphs. While Kida's formula is basically a formula for $\Z_p$-towers, we consider $\Z_p^d$-towers instead.

\begin{thm}[Theorem \ref{thm:weighted_kida_formula}]\label{MainThm2}
Let $\{X_n\}_{n\geq 0}$ be a $\Z_p^d$-tower over a weighted symmetric digraph $X$. Let $Y\slash X$ be a Galois cover of weighted symmetric digraphs of degree $p$-power. Let $\{Y_n\}_{n\geq 0}$ be a $\Z_p^d$-tower over $Y$ compatible with $\{X_n\}_{n\geq 0}$. Let $\mu(X), \mu(Y)$ denote the Iwasawa $\mu$-invariants of $\{X_n\}, \{Y_n\}$ and let $\lambda(X), \lambda(Y)$ denote the Iwasawa $\lambda$-invariants of $\{X_n\}, \{Y_n\}$ respectively. Under the proper assumptions (see Theorem \ref{thm:weighted_kida_formula} for details), we have that $\mu(X)=0$ if and only if $\mu(Y)=0$.
 Moreover, if this equivalent condition is satisfied, then
\begin{align} 
\lambda(Y)=\begin{cases}
[Y:X]\lambda(X)&(\text{if }d\geq 2),\\
[Y:X](\lambda(X)+1)-1&(\text{if }d=1),
\end{cases}
\end{align}
where $[Y:X]$ is the degree of the cover $Y/X$.
\end{thm}
\noindent The work of Ray--Valli\`{e}res \cite[Theorem 4.1]{RV22} can be regarded as Theorem \ref{MainThm2} in the case of trivial weights and $d=1$.

In \cite{MV24}, McGown and Valli\`{e}res established an algorithm for computing the Iwasawa $\mu$ and $\lambda$-invariants for $\Z_p$-towers of graphs numerically. Combining with the calculation method of $\mu$ and $\lambda$-invariants for $\Z_p^d$-towers of number fields constructed by Cuoco--Monsky in \cite{CM81}, we also generalize their algorithm to those of $\Z_p^d$-towers of weighted graphs and provide several examples. In particular, we succeeded in giving an example of Theorem \ref{MainThm2} in the case where $d=2$, $\mu(X)=0$, $[Y:X]=8$, $\lambda(X)=2$, and $\lambda(Y)=16$ in Example \ref{ex:kida}. 

Since the theory of Artin--Ihara $L$-functions of weighted graphs has applications to applied mathematics including quantum walks, we are expecting that Iwasawa theory constructed in this paper will also have some applications. As for quantum walks, see \cite{Por18}, for example. We are wishing that the method of Iwasawa theory developed in number theory in this half a century will have some contribution to applied mathematics. As a first step to such an application, we would like to consider the theory of discrete-time quantum walks in graphs. It is known that the transition matrices of discrete-time quantum walks in graphs are closely related to the Ihara zeta functions of weighted graphs. Making use of the relation, Konno and Sato proved the celebrated Konno--Sato theorem, which has an application to calculate the eigenvalues of the transition matrices {\cite[Corollary 4.2]{KS12}}.

While we have such an application for the eigenvalues of transition matrices, can we say anything about the values of the characteristic polynomial when a non-eigenvalue $a$ is substituted? We have the following our main theorem. 

\begin{thm}[Theorem \ref{thm:quantumwalk}]\label{MainThm3}
Let $\{X_n\}_{n\geq 0}$ be a $\Z_p^d$-tower over a symmetric digraph $X$ and let $\bm{U}_n$ be the transition matrix of a discrete-time quantum walk in $X_n$. Let $a$ be an element of $\bar{\Q}$ (or more generally $\bar{\Q}_p$) such that, for every $n\geq 0$, $a$ is not an eigenvalue of $\bm{U}_n$. Then there exist $\mu\in\Q, \lambda\in\Z_{\geq 0}$, and $\mu_1,\ldots,\mu_{d-1},\lambda_1,\ldots,\lambda_{d-1},\nu\in\Q$ such that, for all $n\gg0$, we have
\[
v_p(\det(a\bm{I}_{2l_n}-\bm{U}_n))=(\mu p^n+\lambda n) p^{(d-1)n}+\left(\sum_{i=1}^{d-1}(\mu_ip^n+\lambda_i n)p^{(d-i-1)n}\right)+\nu,
\]
where $2l_n$ is the number of the directed edges of $X_n$.
\end{thm}
This can be regarded as a graph discrete-time quantum walk analogue of Conjecture \ref{conj:Greenberg}.

Finally, we will prove the following useful theorem on comparison between the weighted complexities in a cover of graphs, on the way to prove Theorem \ref{MainThm1} and Theorem \ref{MainThm2}.
\begin{thm}[Theorem \ref{thm:recurrence_formula}]\label{MainThm4}
Let $X$ be a weighted symmetric digraph. Let $G$ be a finite abelian group and let $X(G,\alpha)$ be a voltage cover defined in Definition \ref{def:voltagecover}. Let $h_X(\psi,1,\alpha)$ be elements in $\bar{\Q}_p$ defined in Definition \ref{def:h}. Let $\widehat{G}$ denote the character group ${\rm Hom}(G, \bar{\Q}^*)$ of $G$. Then it holds that
\begin{align}
\kappa^W(X(G,{\alpha}))=\frac{\kappa^W(X)}{|G|}\prod_{\psi\in\widehat{G}\backslash\{1\}}h_X(\psi,1,\alpha).
\end{align}
\end{thm}
This statement was shown by Mizuno--Sato for (simple) weighted symmetric digraphs tacitly assuming that the total weight does not coincide with the number of vertices in \cite[Corollary 1]{MS03}. The total weight coincides with the number of edges when we consider a graph with trivial weights, and this assumption has always been imposed whenever one makes use of this type of statements. While this assumption was insignificant when we consider a graph with trivial weights since it means the genus of the graph is not $1$, it is not the case for weighted graphs. We succeeded in removing the assumption.

\noindent
\textbf{Organization of the paper.}
In \S\ref{sec:pre}, we prepare basic notions for our proofs. We recall the definitions regarding graphs in \S\ref{sec:graph} and that of voltage covers in \S\ref{sec:voltage}. In \S\ref{sec:Sato_L}, we introduce Mizuno--Sato $L$-function that is a key tool of our proofs. Weighted complexities are explained in \S\ref{sec:weighted_complexities}, which are the main objects studied in this paper.

In \S\ref{sec:ICNF}, we study asymptotic behaviors of weighted complexities in a $\Z_p^d$-tower via Theorem \ref{MainThm4}. We prove a weighted graph-analogue of Iwasawa's class number formula in \S\ref{sec:Iwasawa_type}. 

We prove an analogue of Kida's formula in \S\ref{sec:Kida}, which measures the variation of Iwasawa invariants defined in \S\ref{sec:ICNF} in a Galois cover of $p$-power degree. 

In \S\ref{sec:calculation}, we introduce an algorithm to calculate the characteristic elements of $\Z_p^d$-towers of graphs via the method of \cite[\S5-\S6]{MV24}. 

In \S\ref{sec:examples}, we give examples of our main results proved in \S\ref{sec:ICNF} and \S\ref{sec:Kida} by using the algorithm of calculation introduced in \S\ref{sec:calculation}. In particular, we calculate characteristic elements, Iwasawa $\mu$-invariants, and Iwasawa $\lambda$-invariants.

Finally, in \S\ref{sec:quantum_walks}, we apply the ideas of Iwasawa theory and the theory of weighted graphs to study transition matrices of discrete-time quantum walks in graphs.\\

\noindent
\textbf{Acknowledgements.}
The authors would like to thank Ryosuke Murooka for helping us to consider the proof of Theorem \ref{MainThm4} together. The authors are grateful to Takenori Kataoka and Iwao Sato for reading an earlier version of this paper carefully and pointing out mathematical mistakes. The authors would also like to express their gratitude to Akihiro Goto, Takashi Hara, Ayaka Ishikawa, Yu Katagiri, Shinichi Kobayashi, Norio Konno, Antonio Lei, Sudipta Mallik, Masanori Morishita, Yusuke Nakamura, Anwesh Ray, Jun Ueki, Daniel Valli\`{e}res, and the anonymous referee for their helpful comments and encouragement. The first author was supported by WISE program (MEXT) at Kyushu University. 

\section{Preliminaries}\label{sec:pre}
\subsection{Graph theory}\label{sec:graph}
We recall basic notions of graph theory in this section. Our references are \cite[\S2.1]{Ser03}, \cite{Nor98}, \cite{GR01}, \cite{Bel98}, and \cite[\S1]{MS03}. 
\begin{defi} 
First, we give the definitions of the following terms.
\begin{itemize}
\item An \emph{undirected graph} $X$ is a data $(V(X),E(X),\Phi)$, where $V(X)$ is a set of vertices of $X$, $E(X)$ is a set of undirected edges of $X$, and $\Phi:E(X)\to V(X)^{(2)}$ is a map. Here, $V(X)^{(2)}$ is the set of unordered pairs of $V(X)$. For $e\in E(X)$, we call $v_1$ and $v_2$ the \emph{ends} of $e$ if $\Phi(e)=\{v_1,v_2\}$. 
\item A \emph{directed graph (or digraph)} $Y$ is a data $(V(Y),\E(Y),\inc:\E(Y)\ni e\mapsto (o(e),t(e))\in V(Y)\times V(Y), \E_{\rm{inv}}(Y), \overline{\cdot}\in\auto(\E_{\rm{inv}}(Y))$, where $V(Y)$ is a set of vertices of $Y$, $\E(Y)$ is a set of directed edges of $Y$, and $\E_{\rm{inv}}(Y)$ is a subset of $\E(Y)$ satisfying $\bar{e}\neq e, \bar{\bar{e}}=e$, and $o(e)=t(\bar{e})$ for all $e\in \E_{\rm{inv}}(Y)$. The map $\inc:\E(Y)\to V(Y)\times V(Y)$ is called the \emph{incident map}, $o(e)$(resp. $t(e)$) is called an \emph{original point} (resp. a \emph{terminal point}) \emph{of} $e$. An element $e\in\E(X)$ is called a \emph{directed edge} from $o(e)$ to $t(e)$. For $e\in \E_{\rm{inv}}(Y)$, $\bar{e}$ is called the \emph{inverse edge of} $e$.
\item An undirected (resp. A directed) graph $X$ is said to be \textit{finite} if $V(X)$ and $E(X)$ (resp. $\E(X)$) are finite sets. 
\end{itemize}
\end{defi}
In this paper, we allow graphs to have multi-loops and multi-edges, and unless otherwise mentioned, we assume that all graphs are finite. 
\begin{defi}
Let $X$ be an undirected graph. A sequence $C=e_1e_2\cdots e_{n-1}$ consisting of distinct edges $e_i\in E(X)$ is called a \emph{trail} in $X$ if there exists a sequence of vertices $\{v_i\}_{1\leq i\leq n}$ such that $\Phi(e_i)=\{v_i,v_{i+1}\}$ for every $1\leq i\leq n-1$. A trail $C=e_1e_2\cdots e_{n-1}$ is a \emph{cycle} if $v_1=v_n$ and $v_i\neq v_j$ for all $i,j$ such that $i\neq j$ and $\{i,j\}\neq\{1,n\}$.
\end{defi}
We give some definitions needed when constructing Mizuno--Sato $L$-functions in \S\ref{sec:Sato_L}.
\begin{defi}
Let $X$ be a directed graph.
\begin{itemize}
\item A \textit{path} in $X$ is a sequence $C=e_1e_2\cdots e_n$, $e_i\in\E(X)$ such that $t(e_i)=o(e_{i+1})$ for all $1\leq i\leq n-1$. We define $|C|:=n$ and call it the \emph{length} of $C$.
\item A path $C=e_1e_2\cdots e_n$ has a \textit{backtrack} if $e_{i+1}=\bar{e}_{i}$ for some $1\leq i \leq n-1$.
\item  A path $C=e_1e_2\cdots e_n$ has a \textit{tail} if $e_{1}=\bar{e}_{n}$.
\item A path $C=e_1e_2\cdots e_n$ is a \textit{closed path} if $o(e_1)=t(e_n)$.
\item A closed path $C=e_1e_2\cdots e_n$ is a \textit{cycle} if it satisfies the following conditions: \begin{enumerate}
    \item It has no backtracks,
    \item It has no tails,
    \item $t(e_i)\neq t(e_j)$ for all $1\leq i,j \leq n$ with $i\neq j$.
\end{enumerate}
\item Denote the closed path obtained by going $r$-times around $C$ by $C^r$ ($r\in\Z_{>0}$). 
\item A closed path $C$ is \textit{prime} if $C\neq D^r$ for any $r\geq 2$ and any closed path $D$ in $X$. (We adopt this definition to make use of \cite[Theorem 9]{MS14}.)
\item Two closed paths $C_1=e_1e_2\cdots e_n$ and $C_2=e'_1e'_2\cdots e'_n$ are \textit{equivalent} if there exists an integer $k$ such that, for all $1\leq j\leq n$,
\begin{align}
e'_j=\begin{cases}
e_{j+k}&(j+k\leq n),\\
e_{j+k-n}&(j+k>n).
\end{cases}
\end{align}
Denote the equivalence class of a closed path $C$ by $[C]$.
\item $X$ is \emph{connected} if there exists a path $C=e_1e_2\cdots e_r$ such that $o(e_1)=v_1$ and $t(e_r)=v_2$ for arbitrary $v_1,v_2\in V(X)$. 
\end{itemize}
\end{defi}
We recall the definition of the \textit{spanning trees} of graphs that have been studied in previous studies on Iwasawa theory of graphs.
\begin{defi} 
Let $X$ be an undirected graph and $Y$ a directed graph.
\begin{enumerate}[label=$\bullet$]
\item An undirected graph $X'$ is called a \emph{subgraph of }$X$ if $V(X')\subseteq V(X)$, $E(X')\subseteq E(X)$, and $\Phi_{X'}=\Phi_X|_{E(X')}$. Similarly, a directed graph $Y'$ is called a \textit{subgraph of }$Y$ if $V(Y')\subseteq V(Y)$, $\E(Y')\subseteq \E(Y)$, and $\inc_{Y'}=\inc_Y|_{\E(Y)}$.
\item A subgraph $X'$ of $X$ is called a \emph{spanning} subgraph if $V(X')=V(X)$ Similarly, A subgraph $Y'$ of $Y$ is called a \textit{spanning} subgraph if $V(Y')=V(Y)$.
\item A \textit{spanning tree} of $X$ is a connected spanning subgraph of $X$ without cycles.
\item A \emph{spanning arborescence} of $Y$ \emph{rooted at} $v_0\in V(Y)$ is a spanning subgraph without cycles such that, for each vertex $v\in V(Y)$ other than $v_0$, there uniquely exists a path from $v_0$ to $v$.
\end{enumerate}
\end{defi}
\begin{defi}
Let $\Q$ (resp. $\Q_p$) be the field of rationals (resp. $p$-adic numbers). Fix two algebraic closures $\bar{\Q}$ (resp. $\bar{\Q}_p$) of $\Q$ (resp. $\Q_p$) and an embedding $\bar{\Q}\hookrightarrow \bar{\Q}_p$. Also, fix an embedding $\bar{\Q}\hookrightarrow \mathbb{C}$. A \emph{weighted undirected graph} $X$ is a graph $X$ with a function $w:E(X)\to\bar{\Q}_p$ called a \emph{weight function}. A \emph{weighted digraph} $X$ is a directed graph $X$ with a function $w:\E(X)\to\bar{\Q}_p$ called a \emph{weight function}. We remark that, although a lot of our references consider the situation where the codomain of the weight function is $\mathbb{C}$ instead of $\bar{\Q}_p$, their arguments also work for $\bar{\Q}_p$.
\end{defi}
\begin{defi}[{\cite[\S2.1]{Ser03}}]
A \emph{symmetric digraph} $X$ is a digraph satisfying $\E(X)=\E_{\rm{inv}}(X)$.
We define the \textit{Euler characteristic} of $X$ by
\begin{align}
\chi(X):=\#V(X)-\frac{1}{2}\#\E(X).
\end{align}
\end{defi}
\begin{rem}
In Serre's book \cite{Ser03} and previous studies on Iwasawa theory of graphs, the term ``graphs'' means symmetric digraphs.
\end{rem}
For an arbitrary undirected graph $X$, by giving a direction of $e\in E(X)$ and the inverse edge $\bar{e}$ of $e\in E(X)$, we can construct the \emph{symmetric digraph} $D_X$ \emph{of} $X$. In this sense, considering an undirected graph is essentially equivalent to considering a symmetric digraph. 

It is known that an undirected graph has at least one spanning tree if and only if it is connected (cf. \cite[Theorem 4.3.2]{BR12}). In \S\ref{sec:weighted_complexities}, we introduce the \textit{weighted complexities} of weighted symmetric digraphs that are generalization of spanning trees.
\begin{defi}
A \emph{weighted symmetric digraph} is a symmetric digraph $X$ equipped with a weight function $w=w_X:\E(X)\to\bar{\Q}_p$.

Let $X$ be a weighted symmetric digraph with $V(X)=\{v_1,v_2,\ldots,v_m\}$. We define the \emph{weighted matrix} $\bm{W}(X):=(w_{ij})_{i,j}$ by
\begin{align}
 w_{ij}:=\sum_{\substack{e\in\E(X)\\ o(e)=v_i, t(e)=v_j}}w(e). 
\end{align}
A weighted matrix $\bm{W}(X)$ is \textit{strongly-symmetric} if $w(e)=w(\bar{e})$ for every edge $e\in\E(X)$.

A \emph{weighted symmetric digraph $D_X$ derived from a weighted undirected graph $X$} is a weighted symmetric digraph equipped with a weight function $w:\E(D_X)\to\bar{\Q}_p$ with strongly-symmetric weighted matrix naturally derived from the weight function of $X$.
\end{defi}
Hence considering a weighted undirected graph is essentially equivalent to considering a weighted symmetric digraph with strongly-symmetric weighted matrix. 
\begin{rem}
In \cite{MS14}, the term ``\textit{pseudo-symmetric}'' is adopted instead of strongly-symmetric.
\end{rem}
We recall the definitions of morphisms and coverings of weighted graphs. For more details, we refer to \cite[\S2.3]{DV23} and \cite[\S2]{Kat24}.
\begin{defi}
Let $X=(V(X),\E(X),w_X)$ and $Y=(V(Y),\E(Y),w_Y)$ be weighted symmetric digraphs. A \textit{morphism} of weighted graphs $f:X\to Y$ is a pair $(f_{V}:V(X)\to V(Y),f_{\E}:\E(X)\to \E(Y))$ of maps such that
\begin{itemize}
\item $f_{V}(o(e))=o(f_{\E}(e))$ and $f_{V}(t(e))=t(f_{\E}(e))$ for every $e\in \E(X)$,
\item $\overline{f_{\E}(e)}=f_{\E}(\bar{e})$ for every $e\in \E(X)$,
\item $w_Y(f_{\E}(e))=w_X(e)$ for every $e\in \E(X)$.
\end{itemize}
\end{defi}
\begin{defi}\label{def:cover}
Let $X$ and $Y$ be connected and weighted symmetric digraphs. 
\begin{itemize}
\item A morphism $\pi:Y\to X$ is called a (\emph{unramified}) \emph{covering} or merely a \emph{cover} if $\pi$ satisfies the following conditions:
\begin{enumerate}
\item $\pi_V$ and $\pi_{\E}$ are surjective,
\item For each $v\in V(Y)$, the restriction of $\pi_{\E}$ to $\E_{Y,v}$
\begin{align}
\pi_{\E}|_{\E_{Y,v}}:\E_{Y,v}\to \E_{X,\pi_V(v)}
\end{align}
is bijective, where $\E_{Y,v}:=\{e\in \E_Y\ |\ o(e)=v\}$ for $v\in V(Y)$. 
\end{enumerate}
\item An \textit{automorphism} of $Y\slash X$ is an automorphism of $Y$ as a morphism of weighted graphs that is compatible with the covering map $\pi:Y\to X$. 
\item A covering $Y\slash X$ is said to be \textit{Galois} (or \textit{normal}) if the order of the automorphism group $\autom(Y\slash X)$ of $Y\slash X$ is $\#\pi^{-1}(v)$ for any choice of $v\in V(X)$, and in this case, put the Galois group $\gal(Y\slash X):=\autom(Y\slash X)$.
\end{itemize}
\end{defi}
\subsection{Weighted voltage covers}\label{sec:voltage}
\indent In this subsection, we recall the notion of weighted voltage covers. See also \cite[\S3]{Sat07}. Let $X$ be a connected graph with $m=\#V(X)$, $l=\frac{1}{2}\#\E(X)$, and a weighted matrix $\bm{W}(X)$.
\begin{defi}\label{def:voltageassignment}
For a (possibly infinite) weighted symmetric digraph $X$ and a (possibly infinite) group $G$,  a map $\alpha:\E(X)\to G$ is called a \emph{voltage assignment} if $\alpha(\bar{e})=\alpha(e)^{-1}$ for each $e\in \E(X)$.
\end{defi}
We construct the derived graph $X(G,\alpha)$ for a pair $(X,\alpha:\E(X)\to G)$ where $X$ is a weighted symmetric digraph, $G$ is a finite group, and $\alpha$ is a voltage assignment.
\begin{defi}\label{def:voltagecover}
Set $V(X(G,{\alpha})):=V(X)\times G$ and $\E(X(G,{\alpha})):=\E(X)\times G$. For $(e,\sigma)\in \E(X(G,{\alpha}))$, we define the original point $o((e,\sigma))$ as $(o(e),\sigma)$ and the terminal point $t((e,\sigma))$ as $(t(e), \sigma\alpha(e))$. We also define the inverse edge of $(e,\sigma)$ by $\bar{(e,\sigma)}:=(\bar{e},\sigma\alpha(e))$. We equip a weight function $w:\E(X(G,{\alpha}))\to \bar{\Q}_p$ with $X(G,{\alpha})$ by $w((e,\sigma)):=w(e)$. It is easy to check that $X(G,\alpha)$ is a weighted symmetric digraph.
\end{defi}
 The following facts are known (cf. \cite[Theorem 8]{Gon21}). The natural covering $\pi:X(G,{\alpha})\to X$ defined by a pair of maps $\pi_V(v,\sigma):=v$ and $\pi_{\E}(e,\sigma):=e$ for $v\in V(X)$, $e\in\E(X)$ and $\sigma\in G$ becomes a Galois covering such that its automorphism group coincides with $G$ when $X(G,\alpha)$ is connected. Conversely,  every finite Galois cover $Y\slash X$ is isomorphic to a voltage cover whose Galois group is equal to $\gal(Y\slash X)$. We can also prove a similar result for weighted symmetric digraphs by making slight modifications to the argument of \cite[Theorem 8]{Gon21}. The weighted analogue of \cite[Theorem 4]{Gon21} also gives a necessary and sufficient condition for $X(G,{\alpha})$ to be connected.
\subsection{Weighted Artin--Ihara $L$-functions}\label{sec:Sato_L}
In this subsection, we explain weighted Artin--Ihara $L$-functions, which will be used in the proof of our main results. There are various types of weighted Artin--Ihara $L$-functions. We will utilize the following $L$-functions first discussed by Mizuno--Sato in \cite{MS04}. Let $X$ be a weighted symmetric digraph with $\mathbb{E}(X)=\{e_1,e_2,\ldots,e_l,e_{l+1}=\overline{e}_1,\ldots, e_{2l}=\overline{e}_l\}$. Let $G$ be a finite group and let $\alpha:\mathbb{E}(X)\to G$ be a voltage assignment. To construct weighted $L$-functions, for each $\sigma\in G$, we define two matrices $\bm{B}_{\sigma}:=(b_{ij}^{(\sigma)})_{i,j}$ and $\bm{C}_{\sigma}:=(c_{ij}^{(\sigma)})_{i,j}$ by
\begin{align}
b_{ij}^{(\sigma)}:=\begin{cases}
w(e_j)&(\text{if }\alpha(e_i)=\sigma\text{ and }t(e_i)=o(e_j)),\\
0&(\text{otherwise}),
\end{cases}
\ c_{ij}^{(\sigma)}:=\begin{cases}
1&(\text{if }\alpha(e_i)=\sigma\text{ and }\bar{e_i}=e_j),\\
0&(\text{otherwise}).
\end{cases}
\end{align}
In this paper, if we merely say a representation, then we mean that it is a representation on a $\bar{\Q}_p$-vector space.

Let $\bar{\Q}_p\llbracket t \rrbracket$ denote the formal power series ring.
\begin{defi}[{\cite{Sat07}}, {\cite{MS04}} with $\mathbb{C}$ replaced by $\bar{\Q}_p$]\label{def:Sato'L}
Let $\rho$ be a representation of $G$ with degree $d_{\rho}\ (\geq1)$. \emph{Mizuno--Sato $L$-function} associated with $\rho$ and $\alpha$ is defined by
\begin{align}
L(X,t,\rho,\alpha):=\det\left(\bm{I}_{2ld_{\rho}}-t\sum_{\sigma\in G}(\bm{B}_{\sigma}-\bm{C}_{\sigma})\bigotimes \rho(\sigma)\right)^{-1}\in \bar{\Q}_p\llbracket t \rrbracket .
\end{align}
Here $\bigotimes$ is the Kronecker product of matrices and $\bm{I}_{2ld_{\rho}}$ is the $2ld_{\rho}\times2ld_{\rho}$ identity matrix. In particular, for the trivial representation $1_{\rm{rep}}$ of $G$ and the trivial voltage assignment $1_{\rm{vol}}:\E(X^{\alpha})\to G$, we denote $L(X,t,1_{\rm{rep}},1_{\rm{vol}})$ by $\zeta(X,t)$ and call it a \emph{zeta function} of $X$. Note that the constant term of the formal power series $\det\left(\bm{I}_{2ld_{\rho}}-t\sum_{\sigma\in G}(\bm{B}_{\sigma}-\bm{C}_{\sigma})\bigotimes \rho(\sigma)\right)$ is $1$, which implies it is invertible in $\bar{\Q}_p\llbracket t\rrbracket$.
\end{defi}
Now, we observe some properties of Mizuno--Sato $L$-functions, which play crucial roles in our proofs. 
\begin{thm}[{\cite[Corollary 1]{Sat07}} with $\mathbb{C}$ replaced by $\bar{\Q}_p$]\label{thm:Sato_prod}
Under the same setting as Definition \ref{def:Sato'L}, we have
\begin{align}\label{eq:Sato_prod}
\zeta(X(G,{\alpha}),t)=\prod_{\rho}L(X,t,\rho,\alpha)^{d_{\rho}},
\end{align}
where $\rho$ runs over all isomorphism classes of irreducible representations of $G$.
\end{thm}
Mizuno--Sato $L$-functions satisfy the three-term determinant formula as well. 
\begin{thm}[Three-term determinant formula, {\cite[Theorem 6]{Sat07}} with $\mathbb{C}$ replaced by $\bar{\Q}_p$]\label{thm:weighted_L_3term}
Under the same setting as Definition \ref{def:Sato'L}, we have
\begin{align}
&L(X,t,\rho,\alpha)^{-1}\\
&=(1-t^2)^{-\chi(X)d_{\rho}}\det\left(\bm{I}_{md_{\rho}}-t\sum_{\sigma\in G}\bm{W}(\sigma)\bigotimes\rho(\sigma)+t^2\left((\bm{D}^W(X)-\bm{I}_m)\bigotimes\bm{I}_{d_{\rho}} \right)\right),
\end{align}
where $\bm{D}^W(X):=(d_{ij}^W)_{ij}$ with
\begin{align}
d_{ij}^W:=\begin{cases}
\displaystyle\sum_{\substack{e\in \E(X)\\ o(e)=v_i}}w(e)&(\text{if }i=j),\\
0&(\text{otherwise}),
\end{cases}
\end{align}
and, for $\sigma\in G$, $\bm{W}(\sigma)=\bm{W}_{\alpha}(\sigma):=(w_{ij}^{(\sigma)})_{ij}$ with
\begin{align}
w_{ij}^{(\sigma)}:=\displaystyle\sum_{\substack{e\in\E(X),\alpha(e)=\sigma\\ o(e)=v_i,t(e)=v_j}} w(e).
\end{align}
\end{thm}
Note that we have $\bm{W}(X)=\sum_{\sigma\in G}\bm{W}(\sigma)$.

Also, Mizuno--Sato $L$-functions have the forms of Euler products.
\begin{thm}[{\cite[Theorem 9]{MS14}} with $\mathbb{C}$ replaced by $\bar{\Q}_p$]\label{thm:Sato_Euler}
Under the same setting as Definition \ref{def:Sato'L}, we have
\begin{align}\label{eq:Euler_prod}
L(X,t,\rho,\alpha)=\prod_{[C]}\det\left(\bm{I}_d-\rho(\alpha(C))b_w(C)t^{|C|}\right)^{-1},
\end{align}
where $[C]$ runs over all equivalence classes of prime paths of $X$, and
\begin{align}
C=e_1e_2\cdots e_r,\ \alpha(C):=\prod_{i=1}^r\alpha(e_i),\ b_w(C):=\prod_{i=1}^r(w(e_{i+1})-\delta_{\bar{e}_ie_{i+1}}).
\end{align}
Here $\delta$ stands for the Kronecker delta.
\end{thm}
\begin{rem}
The right-hand side of the equation \eqref{eq:Euler_prod} is well-defined even when $G$ is not abelian. In this case, although $\alpha(C_1)$ and $\alpha(C_2)$ are not necessarily equal for two prime paths $C_1$ and $C_2$ such that $[C_1]=[C_2]$, we find $\alpha(C_1)$ and $\alpha(C_2)$ are conjugate to each other. Thus we see that the determinant in \eqref{eq:Euler_prod} does not depend on the choice of a representative in the equivalence class.
\end{rem}
We prove the following theorems by using the Euler product-form of Mizuno--Sato $L$-functions.
\begin{thm}\label{thm:property_SatoL}
Under the same setting as Definition \ref{def:Sato'L}, we have the followings.

(i) Let $\rho_i$ $(i=1,2)$ be representations of $G$. Then we have
\begin{align}\label{thm:sum_prod}
L(X,t,\rho_1\oplus\rho_2,\alpha)=L(X,t,\rho_1,\alpha)L(X,t,\rho_2,\alpha).
\end{align}

(ii) Let $H$ be a subgroup of $G$, $\tilde{X}\slash X$ the intermediate cover corresponding to $H$, and $\beta:\bm{E}(\tilde{X})\to H$ a voltage assignment corresponding to $X(G, \alpha)\slash \tilde{X}$ (See \cite[Theorem 8]{Gon21}). If $\rho$ is a representation of $H$, then it holds that
\begin{align}\label{eq:Sato_Induction}
L(X,t,\indu_H^G\rho,\alpha)=L(\tilde{X},t,\rho,\beta). 
\end{align}  
\end{thm}
\begin{proof}
(i) By imitating the proof of \cite[Theorem 19.4]{Ter11}, we can easily prove (i). Indeed, by combining Theorem \ref{thm:Sato_Euler} with $\tra \log A=\log \det A$ for a matrix $A$ (cf. \cite[Exercise 4.1]{Ter11}), we have
\begin{align}
&\log L(X,t,\rho_1\oplus\rho_2,\alpha)\\
&=-\sum_{[C]}\tra\log\left(\bm{I}_{d_{\rho_1\oplus\rho_2}}-(\rho_1\oplus\rho_2)(\alpha(C)))b_w(C)t^{|C|}\right)\\
&=\sum_{C}\sum_{j\geq 1}\frac{1}{|C^j|}\tra((\rho_1\oplus\rho_2)(\alpha(C^j)))b_w(C^j)t^{|C^j|}\\
&=\sum_{C}\sum_{j\geq 1}\frac{1}{|C^j|}\tra(\rho_1(\alpha(C^j)))b_w(C^j)t^{|C^j|}+\sum_{C}\sum_{j\geq 1}\frac{1}{|C^j|}\tra(\rho_2(\alpha(C^j)))b_w(C^j)t^{|C^j|}\\
&=-\sum_{[C]}\tra\log\left(\bm{I}_{d_{\rho_1}}-(\rho_1(\alpha(C)))b_w(C)t^{|C|}\right)-\sum_{[C]}\tra\log\left(\bm{I}_{d_{\rho_2}}-(\rho_2((\alpha(C)))b_w(C)t^{|C|}\right)\\
&=\log L(X,t,\rho_1,\alpha)+\log L(X,t,\rho_2,\alpha)\\
&=\log (L(X,t,\rho_1,\alpha) L(X,t,\rho_2,\alpha)),
\end{align}
where $C$ runs over all prime paths (resp. equivalence classes of prime paths) of $X$ in $\sum_C$ (resp. $\sum_{[C]}$). Here, we used the equality $\tra(\rho_1\oplus\rho_2)=\tra\rho_1+\tra\rho_2$.

(ii) Note that $\alpha(C)$ in Theorem \ref{thm:Sato_Euler} is the normalized Frobenius element in the sense of \cite[Definition 16.1]{Ter11}. Since we have an Euler product-form for Mizuno--Sato $L$-functions by Theorem \ref{thm:Sato_Euler}, we can apply the proof of a similar proposition for edge Artin $L$-functions in \cite[Theorem 19.4(2),(4)]{Ter11}. 
\end{proof}
\subsection{Weighted complexities}\label{sec:weighted_complexities}
In this subsection, we recall the definition of \textit{weighted complexities}, which is a generalization of the number of spanning trees. Reference is \cite{MS03}. In subsequent sections, we investigate the asymptotic behaviors of the weighted complexities in $\Z_p^d$-towers of weighted graphs.
\begin{defi}
Let $X$ be a connected and weighted symmetric digraph with symmetric weighted matrix. We define the \emph{weighted comlexity} $\kappa^W(X)$ of $X$ by
\begin{align} 
\kappa^W(X):=\sum_T\left(\prod_{e\in \E(T)} w(e)\right),
\end{align}
where $T$ runs over the spanning arborescences of $X$ rooted at some fixed $v_0\in V(X)$ and $e$ runs over the edges of $T$ (See \cite[p.19]{MS03} for the cases of simple graphs). By the symmetry of the weighted matrix, we find that $\kappa^W(X)$ is independent of the choice of $v_0$ and is well-defined. If $X$ is a weighted undirected graph and $D_X$ is a weighted symmetric digraph derived from $X$, then define $\kappa^W(X):=\kappa^W(D_X)$. $\kappa^W(X)$ is also called the \emph{weighted complexity} of $X$. If $X$ is a weighted undirected graph with trivial weights, then, by definition, $\kappa^W(X)$ coincides with the number of the spanning trees of $X$.
\end{defi}
By the definition of weighted complexities, we have the following example immediately.
\begin{ex}
We fix an integer $R\geq2$ and put $\zeta_R:=\exp(2\pi\sqrt{-1}\slash R)\in\bar{\Q}$. Let $X$ be a weighted undirected graph with weight function $w$. We assume that the image of $w$ is in $\{\zeta_R^r\in \bar{\Q}\ |\ r=0,\ldots,R-1\}$. We put, for each spanning tree $T$ of $X$,
\begin{align}
w(T):=\prod_{e\in E(T)}w(e).
\end{align}
Then we have
\begin{align}
\kappa^W(X)=\sum_{r=0}^{R-1}\zeta_R^r\cdot\kappa^W(X,r),
\end{align}
where $\kappa^W(X,r):=\#\{T\text{ is a spanning tree of }X\ |\ w(T)=\zeta_R^r\}$. In particular, if $R=2$, then
\[
\kappa^W(X)=(\mbox{the number of positive spanning trees})-(\mbox{the number of negative spanning trees}).
\]
\end{ex}
For a matrix $\bm{A}$, Let $\bm{A}(k_1,k_2)$ denote the submatrix of $\bm{A}$ with row $k_1$ and column $k_2$ deleted. If 
$k_1$ and $k_2$ coincide, i.e., $k_1=k=k_2$, then we merely denote $\bm{A}(k_1,k_2)$ by $\bm{A}(k)$.
\begin{thm}[{\cite[p. 319]{Cha82}}, see also {\cite[\S5]{Orl78}}, \cite{Tut48}, and \cite{Mal24}]\label{thm:weighted_matrix_tree}
Let $X$ be a weighted symmetric digraph. Assume that the weighted matrix $\bm{W}(X)$ is symmetric. Then
\[
\det{((\bm{D}^W(X)-\bm{W}(X))(k_1,k_2))}=(-1)^{k_1+k_2}\kappa^W(X).
\]
for every $1\leq k_1,k_2\leq m$. In particular, we have
\[
\det{((\bm{D}^W(X)-\bm{W}(X))(k))}=\kappa^W(X)
\]
for every $1\leq k\leq m$.
\end{thm}
At the end of this section, we prepare a lemma for a proof in the subsequent section.
\begin{lem}\label{lem:permutationmatrix}
Let $K$ be a field. Let $\{\bm{e}_1,\ldots,\bm{e}_N\}$ be the standard basis of $K^N$ consisting of column vectors. Let $\bm{A}$ be an $N\times N$ matrix over $K$. Let $\bm{R}$ be an $N\times N$ permutation matrix over $K$ associated with a permutation $\varepsilon$ of $\{1,2,\ldots,N\}$, that is, $\bm{R}=(\bm{e}_{\varepsilon(1)},\ldots, \bm{e}_{\varepsilon(N)})^{\mathsf{T}}$. Put $i_0:=\varepsilon^{-1}(1)$. Then we have
\[
\det((\bm{R}^{\mathsf{T}}\bm{A}\bm{R})(1))=\det(\bm{A}(i_0)).
\]
\end{lem}
\begin{proof}
Since $\bm{R}^{\mathsf{T}}=(\bm{e}_{\varepsilon(1)},\ldots,\bm{e}_{\varepsilon(N)})$, we find that, for $1\leq k_1\leq N$ with $\varepsilon(k_1)\neq 1$, the matrix $\bm{R}^{\mathsf{T}}$ with $k_1$-th column deleted has a row consisting of only $0$ entries. Therefore, we have
\[
\det(\bm{R}^{\mathsf{T}}(1,k_1))=
\begin{cases}
\pm 1&\mbox{if }\varepsilon(k_1)=1,\\
0&\mbox{if }\varepsilon(k_1)\neq 1.
\end{cases}
\]
Likewise, we have
\[
\det(\bm{R}(k_2,1))=
\begin{cases}
\pm 1&\mbox{if }\varepsilon(k_2)=1,\\
0&\mbox{if }\varepsilon(k_2)\neq 1.
\end{cases}
\]
Since $i_0=\varepsilon^{-1}(1)$, we have
\[
\det(\bm{R}^{\mathsf{T}}(1,i_0))=\pm 1=\det(\bm{R}(i_0,1)).
\]
By the Cauchy--Binet theorem, we obtain
\begin{align}
\det((\bm{R}^{\mathsf{T}}\bm{A}\bm{R})(1))&=\sum_{1\leq k_1,k_2\leq N} \det(\bm{R}^{\mathsf{T}}(1,k_1))\det(\bm{A}(k_1,k_2))\det(\bm{R}(k_2,1))\\
&=\det(\bm{R}^{\mathsf{T}}(1,i_0))\det(\bm{A}(i_0))\det(\bm{R}(i_0,1))\\
&=\det(\bm{A}(i_0)).
\end{align}
\end{proof}
\section{Weighted Iwasawa's class number formula}\label{sec:ICNF}
In this section, we prove our first and fourth main results.
\subsection{Recurrence relation of weighted complexities in Galois coverings of weighted graphs}
In this subsection, we generalize a result of Mizuno--Sato \cite[Corollary 1]{MS03}, which is about a recurrence relation of weighted complexities for a voltage cover. First, we impose the following assumption throughout this paper except for \S\ref{sec:quantum_walks}.
\begin{assump}\label{assum:symmetric}
The weighted matrix of every weighted symmetric digraph is symmetric.
\end{assump}
This assumption is reasonable. In fact, we can construct an example of a covering $Y\slash X$ of weighted symmetric digraphs such that $\bm{W}(X)$ is symmetric but $\bm{W}(Y)$ is not symmetric. Let $X$ be a weighted symmetric digraph described as follows.
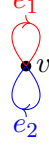
\begin{equation*}
\begin{minipage}[c]{0.3\columnwidth}
\centering
\begin{tikzpicture}
 \fill (0,0) circle[radius=2pt] node[right] {$v$};
\draw[red,decoration={markings,mark=at position \halfway with \arrow{>}},postaction=decorate] (0,0)  to [in=50,out=130,loop] node[above]{$e_1$} (0,0);
\draw[blue,decoration={markings,mark=at position \halfway with \arrow{>}},postaction=decorate] (0,0) to [in=230,out=310,loop] node[below]{$e_2$} (0,0);
\end{tikzpicture}
\end{minipage}
\end{equation*}
We assign the weights by $w(e_1)=a$, $w(\bar{e}_1)=\bar{a}$, $w(e_2)=b$, and $w(\bar{e}_2)=\bar{b}$ ($a$, $\bar{a}$, $b$, and $\bar{b}\in\Z$). Let $\alpha:\E(X)\to \Z\slash 4\Z$ be a voltage assignment satisfying $\alpha(e_1)=\bar{1}$ and $\alpha(e_2)=\bar{2}$. Then we have the following voltage cover $X(\Z\slash 4\Z,\alpha)\slash X$.
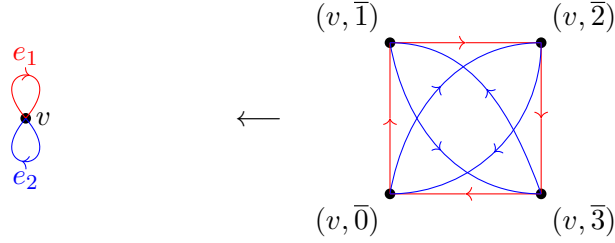
\begin{equation*}
\begin{minipage}[c]{0.3\columnwidth}
\centering
\begin{tikzpicture}
 \fill (0,0) circle[radius=2pt] node[right] {$v$};
\draw[red,decoration={markings,mark=at position \halfway with \arrow{>}},postaction=decorate] (0,0)  to [in=50,out=130,loop] node[above]{$e_1$} (0,0);
\draw[blue,decoration={markings,mark=at position \halfway with \arrow{>}},postaction=decorate] (0,0) to [in=230,out=310,loop] node[below]{$e_2$} (0,0);
\end{tikzpicture}
\end{minipage}
\, \, \, \longleftarrow \, \, \,
\begin{minipage}[c]{0.3\columnwidth}
\begin{tikzpicture}

\fill (0,0) circle[radius=2pt] node[below left] {($v,\bar{0}$)};
\fill (0,2) circle[radius=2pt] node[above left] {($v,\bar{1}$)};
\fill (2,2) circle[radius=2pt] node[above right] {($v,\bar{2}$)};
\fill (2,0) circle[radius=2pt] node[below right] {($v,\bar{3}$)};

\draw[red,->] (0,0) -- (0,1);
\draw[red] (0,1)--(0,2);
\draw[red,->] (0,2) -- (1,2);
\draw[red] (1,2)--(2,2);
\draw[red,->] (2,2)--(2,1);
\draw[red] (2,1)--(2,0);
\draw[red,->] (2,0)--(1,0);
\draw[red] (1,0)--(0,0);

\draw[blue,decoration={markings,mark=at position \halfway with \arrow{>}},postaction=decorate] (0,0) to [out=80,in=180] (2,2); 
\draw[blue,decoration={markings,mark=at position \halfway with \arrow{>}},postaction=decorate] (2,2) to [out=270,in=0] (0,0); 
\draw[blue,decoration={markings,mark=at position \halfway with \arrow{>}},postaction=decorate] (0,2) to [out=280,in=180] (2,0); 
\draw[blue,decoration={markings,mark=at position \halfway with \arrow{>}},postaction=decorate] (2,0) to [out=110,in=0] (0,2); 

\end{tikzpicture}
\end{minipage}
\end{equation*}
Here, to make the figure easier to see, we only drew the edges above $e_1$ and $e_2$. Red (resp. Blue) edges in $X(\Z\slash4\Z,\alpha)$ represent the edges above $e_1$ (resp. $e_2$). The weighted matrix of $X$ is
\begin{align}
\bm{W}(X)=(a+\bar{a}+b+\bar{b}),
\end{align}
which is symmetric. But the weighted matrix of $X(\Z\slash4\Z,\alpha)$ is
\begin{align}
\bm{W}(X(\Z\slash4\Z,\alpha))=\begin{pmatrix}
0&\bar{a}&b+\bar{b}&a\\
a&0&\bar{a}&b+\bar{b}\\
b+\bar{b}&a&0&\bar{a}\\
\bar{a}&b+\bar{b}&a&0
\end{pmatrix}, 
\end{align}
which is not symmetric since $a\neq \bar{a}$ in general.

That being said, Assumption \ref{assum:symmetric} is satisfied for every weighted symmetric digraph $X$ and its covers if we impose that $\bm{W}(X)$ is strongly-symmetric.
\begin{defi}\label{def:h}
For a weighted symmetric digraph $X$, a finite abelian group $G$, a voltage assignment $\alpha:\mathbb{E}(X)\to G$, and a character $\psi\in\widehat{G}:={\rm Hom}(G,\bar{\Q}^*)$, we define
\begin{align}
h_X(\psi,t,\alpha):=\det(\bm{I}_m-t\sum_{\sigma\in G}\bm{W}(\sigma)\bigotimes\psi(\sigma)+t^2(\bm{D}^W(X)-\bm{I}_m)).
\end{align}
Note that $h_X(\psi,t,\alpha)$ is a part of the three-term determinant expression of the associated Mizuno--Sato $L$-function. Put $\bm{W}_{\psi}:=\sum_{\sigma\in G}\bm{W}(\sigma)\bigotimes\psi(\sigma)=\sum_{\sigma\in G}\psi(\sigma)\bm{W}(\sigma)$.

\end{defi}
\begin{thm}\label{thm:recurrence_formula}
Let $X$ be a weighted symmetric digraph, $G$ a finite abelian group, and $\alpha:\mathbb{E}(X)\to G$ a voltage assignment. Then it holds that
\begin{align}
\kappa^W(X(G,{\alpha}))=\frac{\kappa^W(X)}{|G|}\prod_{\psi\in\widehat{G}\backslash\{1_{\widehat{G}}\}}h_X(\psi,1,\alpha).
\end{align}
\end{thm}
\begin{rem}
Mizuno--Sato proved this theorem for (simple) graphs tacitly assuming\\ $\frac{1}{2}\sum_{e\in\mathbb{E}(X)}w(e)\neq m$ in \cite[Corollary 1]{MS03}. This assumption coincides with $\chi(X)\neq 0$ when we consider the case of trivial weights. This type of statement often plays a fundamental role in Iwasawa theory of graphs. See \cite{MV23}, \cite{DV23}, and \cite{RV22}, for example. While this assumption is inconsequential when we consider the case of trivial weights since it is equivalent to assume the genus of $X$ is not equal to zero, it is not the case for weighted graphs.
\end{rem}
\begin{proof}
By the fundamental theorem for finite abelian groups, we may assume that
\[
G=\Z/n_1\Z\oplus\cdots\oplus\Z/n_r\Z.
\]
Put $n:=|G|$. Then $n=n_1\times\cdots\times n_r$. Put
\[
\bm{P}_i:=\begin{pmatrix}
1&1&1&\cdots&1\\
1&\zeta_{n_i}&\zeta_{n_i}^2&\cdots&\zeta_{n_i}^{n_i-1}\\
1&\zeta_{n_i}^2&\zeta_{n_i}^4&\cdots&\zeta_{n_i}^{(n_i-1)2}\\
\vdots&\vdots&\vdots&&\vdots\\
1&\zeta_{n_i}^{n_i-1}&\zeta_{n_i}^{2(n_i-1)}&\cdots&\zeta_{n_i}^{(n_i-1)(n_i-1)}.
\end{pmatrix}
\]
Then $\bm{P}_i$ is invertible and 
\[
\bm{P}_i^{-1}=\frac{1}{n_i}
\begin{pmatrix}
1&1&1&\cdots&1\\
1&\zeta_{n_i}^{n_i-1}&\zeta_{n_i}^{n_i-2}&\cdots&\zeta_{n_i}\\
1&\zeta_{n_i}^{2(n_i-1)}&\zeta_{n_i}^{2(n_i-2)}&\cdots&\zeta_{n_i}^2\\
\vdots&\vdots&\vdots&&\vdots\\
1&\zeta_{n_i}^{(n_i-1)(n_i-1)}&\zeta_{n_i}^{(n_i-1)(n_i-2)}&\cdots&\zeta_{n_i}^{n_i-1}
\end{pmatrix}
\]
since
\begin{eqnarray*}
&&1+\zeta_{n_i}^{k_1}\zeta_{n_i}^{n_i-k_2}+\zeta_{n_i}^{2k_1}\zeta_{n_i}^{2(n_i-k_2)}+\cdots+\zeta_{n_i}^{k_1(n_i-1)}\zeta_{n_i}^{(n_i-k_2)(n_i-1)}\\
&=&1+\zeta_{n_i}^{k_1-k_2}+\zeta_{n_i}^{2(k_1-k_2)}+\cdots +\zeta_{n_i}^{(n_i-1)(k_1-k_2)}\\
&=&
\begin{cases}
n_i&\mbox{if }k_1=k_2,\\
0&\mbox{if }k_1\neq k_2.
\end{cases}
\end{eqnarray*}
Put $G_i:=\Z/n_i\Z$. Note that $\bar{1}\in G_i$ generates $G_i$. Let $\rho_i$ be the right regular representation of $G_i$. Then, for arbitrary $\sigma_i\in G_i$, we can check that
\[
\bm{P}_i^{-1}\rho_i(\sigma_i)\bm{P}_i=(1)\oplus\psi_{i,2}(\sigma_i)\oplus\cdots\oplus\psi_{i,n_i}(\sigma_i),
\]
where $\psi_{i,j}:G_i\to\bar{\Q}^*$ are the irreducible representations such that $\psi_{i,j}(\bar{1})=\zeta_{n_i}^j$. Let $\rho$ be the right regular representation of $G$. Choose $\{(0,\ldots,0), (0,\ldots,0,1)\ldots,(0,\ldots,0,n_r-1),\ldots,(n_1-1,0,\ldots,0),\ldots,(n_1-1,n_2-1,\ldots,n_r-1)\}$ as a basis of the $\bar{\Q}$-vector space $\bar{\Q}[G]$.
Then, by properties of regular representations, for arbitrary $\sigma=(\sigma_1,\ldots,\sigma_r)\in G$, we have
$\rho(\sigma)=\rho_1(\sigma_1)\bigotimes\cdots\bigotimes\rho_r(\sigma_r)$, where $\bigotimes$ denotes the Kronecker product.
Put $\bm{P}:=\bm{P}_1\bigotimes\cdots\bigotimes \bm{P}_r$. By the mixed product property of Kronecker products, we have
\begin{eqnarray*}
\bm{P}^{-1}\rho(\sigma)\bm{P}&=&\bm{P}_1^{-1}\rho_1(\sigma_1)\bm{P}_1\bigotimes\cdots\bigotimes \bm{P}_r^{-1}\rho_r(\sigma_r) \bm{P}_r\\
&=&(1)\oplus\bigoplus_{j=2}^n\psi_j(\sigma),
\end{eqnarray*}
where $\psi_j$ are the irreducible representations of $G$. Put $Y:=X(G,\alpha)$. Then we find that $\bm{W}(Y)=\sum_{\sigma\in G}(\bm{W}(\sigma)\bigotimes\rho(\sigma))$. Put $\bm{L}:=\bm{D}^W(Y)-\bm{W}(Y)$. Since
\begin{eqnarray*}
(\bm{I}_m\bigotimes\bm{P}^{-1}) \bm{W}(Y)(\bm{I}_m\bigotimes \bm{P})&=&(\bm{I}_m\bigotimes \bm{P}^{-1})(\sum_{\sigma\in G}(\bm{W}(\sigma)\bigotimes\rho(\sigma)))(\bm{I}_m\bigotimes \bm{P})\\
&=&\sum_{\sigma\in G}\bm{W}(\sigma)\bigotimes\{(1)\oplus\bigoplus_{j=2}^n\psi_j(\sigma)\},
\end{eqnarray*}
we have
\begin{eqnarray*}
&&(\bm{I}_m \bigotimes \bm{P}^{-1})\bm{L}(\bm{I}_m \bigotimes \bm{P})\\
&=&(\bm{I}_m\bigotimes \bm{P}^{-1})(\bm{I}_{mn}-\bm{W}(Y)+(\bm{D}^W(X)-\bm{I}_m)\bigotimes\bm{I}_n)(\bm{I}_m\bigotimes\bm{P})\\
&=&\bm{I}_{mn}-\sum_{\sigma\in G}\bm{W}(\sigma)\bigotimes\{(1)\oplus\bigoplus_{j=2}^n\psi_{j}(\sigma)\} +(\bm{D}^W(X)-\bm{I}_m)\bigotimes\bm{I}_n\\
&=&\bm{R}^{\mathsf{T}}(\bm{I}_{mn}-\sum_{\sigma\in G}\{(1)\oplus\bigoplus_{j=2}^n\psi_{j}(\sigma)\}\bigotimes\bm{W}(\sigma) +\bm{I}_n\bigotimes(\bm{D}^W(X)-\bm{I}_m))\bm{R}\\
&=&\bm{R}^{\mathsf{T}}((\bm{D}^W(X)-\bm{W}(X))\oplus\bigoplus_{j=2}^n(\bm{I}_m- \sum_{\sigma\in G}\psi_j(\sigma)\bigotimes\bm{W}(\sigma)+\bm{D}^W(X)-\bm{I}_m))\bm{R}\\
&=&\bm{R}^{\mathsf{T}}((\bm{D}^W(X)-\bm{W}(X))\oplus\bigoplus_{j=2}^n \bm{M}(\psi_j,1))\bm{R},
\end{eqnarray*}
where $\bm{M}(\psi_j,t):=\bm{I}_m-t\bm{W}_{\psi_j}+t^2(\bm{D}^W(X)-\bm{I}_m)$ and $\bm{R}$ is an appropriate commutation matrix needed for reversing the orders of our Kronecker products simultaneously. Note that $\bm{R}$ is a special type of permutation matrices. If $\epsilon$ is the permutation of $\{1,\ldots, mn\}$ corresponding to $\bm{R}$, then, by the definition of Kronecker products, we find that $\epsilon(1)=1$. Therefore, by Lemma \ref{lem:permutationmatrix}, we obtain
\begin{eqnarray*}
&&\det(((\bm{I}_m \bigotimes \bm{P}^{-1})\bm{L}(\bm{I}_m \bigotimes \bm{P}))(1))\\
&=&\det((\bm{R}^{\mathsf{T}}((\bm{D}^W(X)-\bm{W}(X))\oplus\bigoplus_{j=2}^n \bm{M}(\psi_j,1))\bm{R})(1))\\
&=&\det(((\bm{D}^W(X)-\bm{W}(X))\oplus\bigoplus_{j=2}^n \bm{M}(\psi_j,1))(1))\\
&=&\det((\bm{D}^W(X)-\bm{W}(X))(1)\oplus\bigoplus_{j=2}^n \bm{M}(\psi_j,1))\\
&=&\kappa^W(X)\prod_{\psi\in\widehat{G}\backslash\{1_{\widehat{G}}\}}h_X(\psi,1,\alpha).
\end{eqnarray*} 
Hence it suffices to show that
\[
\det(((\bm{I}_m\bigotimes\bm{P}^{-1})\bm{L}(\bm{I}_m\bigotimes \bm{P}))(1))=n(\det(\bm{L}(1))).
\]
Put $\bm{Q}:=\bm{I}_m\bigotimes\bm{P}$, $\bm{P}':=n\bm{P}^{-1}$, $\bm{Q'}:=\bm{I}_m\bigotimes\bm{P}'$. By the Cauchy--Binet theorem, we have
\begin{equation}\label{kappaeqn1}
\det((\bm{Q}'\bm{L}\bm{Q})(1))=\sum_{1\leq k_1,k_2 \leq mn}\det(\bm{Q}'(1,k_1))\det(\bm{L}(k_1,k_2))\det(\bm{Q}(k_2,1)).
\end{equation}
Since $\bm{Q}'\bm{Q}=n\bm{I}_{mn}$, we have $\bm{Q}'=\frac{n}{\det \bm{Q}}\widetilde{\bm{Q}}$, where $\widetilde{\bm{Q}}$ is the adjugate matrix of $\bm{Q}$. Hence
\begin{equation}\label{kappaeqn2}
\det(\bm{Q}(k_2,1))=(-1)^{k_2+1}\widetilde{\bm{Q}}_{1,k_2}=\frac{\det \bm{Q}}{n}(-1)^{k_2+1}\bm{Q}'_{1,k_2}.
\end{equation}
By the rudimentary linear algebra, for an $N\times N$ square matrix $\bm{A}$, we have
\[
\det \bm{A}\neq 0\implies\widetilde{(\widetilde{\bm{A}})}=(\det \bm{A})^{N-2}\bm{A}.
\]
Thus
\begin{eqnarray*}
\widetilde{(\bm{Q}')}&=&\widetilde{\left(\frac{n}{\det \bm{Q}}\widetilde{\bm{Q}}\right)}=\left(\frac{n}{\det \bm{Q}}\right)^{mn-1}\widetilde{(\widetilde{\bm{Q}})}\\
&=&\left(\frac{n}{\det \bm{Q}}\right)^{mn-1}(\det \bm{Q})^{mn-2}\bm{Q}=\frac{n^{mn-1}}{\det \bm{Q}}\bm{Q}.
\end{eqnarray*}
Therefore,
\begin{equation}\label{kappaeqn3}
\det(\bm{Q}'(1,k_1))=(-1)^{1+k_1}\widetilde{(\bm{Q}')}_{k_1,1}=(-1)^{1+k_1}\frac{n^{mn-1}}{\det \bm{Q}} \bm{Q}_{k_1,1}.
\end{equation}
On the other hand,
\begin{equation}\label{kappaeqn4}
\det(\bm{L}(k_1,k_2))=(-1)^{k_1+k_2}\kappa^W(Y)=(-1)^{k_1+k_2}\det (\bm{L}(1)).
\end{equation}
By the equations \eqref{kappaeqn1}, \eqref{kappaeqn2}, \eqref{kappaeqn3}, and \eqref{kappaeqn4}, we obtain
\[
\det ((\bm{Q}'\bm{L}\bm{Q})(1))=\sum_{1\leq k_1,k_2\leq mn} n^{(mn-2)}\bm{Q}_{k_1,1}\bm{Q}'_{1,k_2}\det(\bm{L}(1))(-1)^{2k_1+2k_2+2}.
\]
By the definitions of $\bm{P}_i$ and that of Kronecker products, we must have $\bm{Q}_{k_1,1}=1=\bm{Q}'_{1,k_2}$ for all $1\leq k_1,k_2 \leq n$ and $\bm{Q}_{k_1,1}=0=\bm{Q}'_{1,k_2}$ for all $n<k_1,k_2\leq mn$. Therefore,
\[
\det((\bm{Q}'\bm{L}\bm{Q}(1))=n^2\cdot n^{mn-2}\det(\bm{L}(1))=n^{mn}\det(\bm{L}(1)).
\]
Consequently, we conclude that
\begin{eqnarray*}
\det((\bm{Q}^{-1}\bm{L}\bm{Q})(1))&=&\det((\frac{1}{n}\bm{Q}'\bm{L}\bm{Q})(1))\\
&=&\frac{1}{n^{mn-1}}\det((\bm{Q}'\bm{L}\bm{Q})(1))\\
&=&\frac{n^{mn}}{n^{mn-1}}\det(\bm{L}(1))=n\det(\bm{L}(1)).
\end{eqnarray*}
This completes the proof.
\end{proof}
\subsection{Asymptotic theory of Cuoco and Monsky}\label{subsection:CM}
In this subsection, we introduce and slightly develop the theory of Cuoco--Monsky \cite[\S 1]{CM81} and Monsky \cite{Mon81} in order to study the asymptotic behavior of weighted complexities in a $\Z_p^d$-tower in subsequent subsections. See also \cite[\S3.1]{TU24}.

Let $K$ be a finite extension of $\Q_p$, and let $\mathcal{O}=\mathcal{O}_K$ be the valuation ring of $K$. Let $\Lambda_d$ denote the formal power series ring $\mathcal{O}\llbracket T_1,\ldots, T_d \rrbracket$. Let $\Gamma$ be a free $\Z_p$-module of rank $d$ and fix its basis $\{\sigma_1,\ldots,\sigma_d\}$. Let $\mathcal{O}\llbracket \Gamma\rrbracket$ denote the complete group ring. Then we have the Iwasawa--Serre isomorphism $\mathcal{O}\llbracket \Gamma \rrbracket \cong \Lambda_d$ via $\sigma_i-1\mapsto T_i$ for each $1\leq i\leq d$, and so $\Gamma$ can be seen as a closed subgroup of the unit group $\Lambda_d^*$ generated by $\{1+T_i\}_{1\leq i\leq d}$. We see that each element of ${\rm GL}_d(\Z_p)$ defines a $\Z_p$-module automorphism of $\Gamma$ and it also induces a ring automorphism of $\Lambda_d$. If $\sigma\in \Gamma\backslash \Gamma^p$, then we find that there exists an element ${\rm GL}_d(\Z_p)$ such that $\sigma$ maps to $1+T_1$. Let $\mathfrak{m}$ be the maximal ideal of $\mathcal{O}$ and let $\F_q$ be the finite field of order $q$ such that $\F_q\cong \mathcal{O}/\mathfrak{m}$. Put $\Omega_d:=\Lambda_d/\mathfrak{m}\Lambda_d\cong\F_q\llbracket T_1,\cdots,T_d\rrbracket$. Since $(\overline{T}_1)$ is a prime ideal of $\Omega_d$ of height $1$, so is $\mathfrak{p}:=(\overline{\sigma-1})$. Since $\Omega_d$ is a Noetherian unique factorization domain, every prime ideal of height $1$ is principal, and so $\overline{\sigma -1}$ is a prime element of $\Omega_d$. Hence we can define an order function ${\rm ord}_{\mathfrak{p}}:\Omega_d\to\Z_{\geq 0}\cup\{\infty\}$.

Let $\varpi$ be a uniformizer of the discrete valuation ring $\mathcal{O}$. Let $F$ be a non-zero element of $\Lambda_{K,d}:=K\llbracket T_1,\ldots,T_d\rrbracket$ such that
\begin{equation}\label{semipowerseries}
\mbox{there exist }N(F)\in\Z\mbox{ and }F_0\in\Lambda_d\mbox{ such that }F=\varpi^{N(F)}F_0\in\Lambda_d\mbox{ and }\varpi\nmid F_0.
\end{equation}
If such $N(F)$ exists, then we see that it is uniquely determined. Let $e(K/\Q_p)$ be the ramification index of $K/\Q_p$.
\begin{defi}[{A generalization of \cite[Definition 1.1-1.2]{CM81}}]\label{df:CM81_Definition 1.1-2}
Define \emph{the Iwasawa $\mu$-invariant of $F$} by
\[
\mu(F):=\frac{N(F)}{e(K/\Q_p)}\in\Q.
\]
Define \emph{the Iwasawa $\lambda$-invariant of $F$} by
\begin{align}
\lambda(F):=\sum\ord_{\mathfrak{p}}(\overline{F_0})\in\Z_{\geq 0},
\end{align}
where the sum is over $\{\mathfrak{p}=(\overline{\sigma-1})\ |\ \sigma\in \Gamma\backslash \Gamma^p\}$.
\end{defi}
Note that, in particular, if $K=\Q_p$, then $\mu(F)\in \Z$, and if $F\in\Lambda_d$, then $\mu(F)\in\Z_{\geq 0}$.

Let $W_n:=\{\zeta\in \bar{\Q} |\ \zeta^{p^n}=1\}$. For $F\in\Lambda_d$ and $\bm{\zeta}=(\zeta_{(1)},\zeta_{(2)},\ldots,\zeta_{(d)})\in W_n^d$, we define
\[
v_p(F(\bm{\zeta}-\bm{1})):=v_p(F(\zeta_{(1)}-1,\zeta_{(2)}-1,\ldots,\zeta_{(d)}-1)).
\]

\begin{thm}\label{thm:multi_class_number} Let $F$ be a non-zero element of $\Lambda_{K,d}$ satisfying \eqref{semipowerseries}. Suppose that $F(\bm{\zeta-1})\neq0$ for every $\bm{\zeta}\in W^d\setminus\{(1,\ldots,1)\}$. Then there exist $\mu_1,\ldots,\mu_{d-1},\lambda_1,\ldots,\lambda_{d-1},\nu\in\Q$ such that, for all $n\gg 0$, we have
\[
\sum_{\bm{\zeta}\in W_n^d\setminus\{(1,\ldots,1)\}}v_p(F(\bm{\zeta}-\bm{1}))=(\mu(F)p^n+\lambda(F)n)p^{(d-1)n}+\left(\sum_{i=1}^{d-1}(\mu_ip^n+\lambda_in)p^{(d-i-1)n}\right)+\nu.
\]
\end{thm}
\begin{proof}
We have
\begin{align}
    \sum_{\bm{\zeta}\in W_n^d\setminus\{(1,\ldots,1)\}}v_p(F(\bm{\zeta}-\bm{1}))&=\sum_{\bm{\zeta}\in W_n^d\setminus\{(1,\ldots,1)\}} v_p(\varpi^{N(F)}F_0(\bm{\zeta}-\bm{1}))\\
&=\sum_{\bm{\zeta}\in W_n^d\setminus\{(1,\ldots,1)\}} v_p(\varpi^{N(F)})+\sum_{\bm{\zeta}\in W_n^d\setminus\{(1,\ldots,1)\}} v_p(F_0(\bm{\zeta}-\bm{1}))\\
&=\mu(F)(p^{dn}-1)+\sum_{\bm{\zeta}\in W_n^d\setminus\{(1,\ldots,1)\}} v_p(F_0(\bm{\zeta}-\bm{1}))\label{CMeq1}.
\end{align}
For arbitrary $a\in\bar{\Q}_p$, define
\begin{equation*}
v'_p(a):=
\begin{cases}
v_p(a)&\mbox{if } a\neq0,\\
0&\mbox{if } a=0.
\end{cases}
\end{equation*}
Now, the theory of \cite[Section 1]{CM81} can naturally be extended from $\Z_p\llbracket T_1,\ldots, T_d\rrbracket$ to $\Lambda_d$ (except for replacing $``<1"$ in the proof of \cite[Lemma 1.4]{CM81} by $``<\frac{1}{e(K/\Q_p)}"$ and modifying the corresponding part of \cite[Theorem 2.3]{Mon81}). By a natural generalization of {\cite[Theorem 1.7]{CM81}} to $\Lambda_d$ (keeping their unusual convention for the $p$-adic valuation in mind), we have
\begin{align}
\sum_{\bm{\zeta}\in W_n^d\setminus\{(1,\ldots,1)\}} v_p(F_0(\bm{\zeta}-\bm{1}))&=\sum_{\bm{\zeta}\in W_n^d} v'_p(F_0(\bm{\zeta}-\bm{1}))-v'_p(F_0(\bm{0}))\\
&=\lambda(F)np^{(d-1)n}+O(p^{(d-1)n})\ (n\to\infty) \label{CMeq2}
\end{align}
By Monsky {\cite[Theorem 5.6]{Mon81}}, there exist $\mu_0,\ldots,\mu_d,\lambda_0,\ldots,\lambda_d,\nu\in\Q$ such that, for all $n\gg 0$, we have
\begin{equation}\label{CMeq3}
\sum_{\bm{\zeta}\in W_n^d\setminus\{(1,\ldots,1)\}} v_p(F_0(\bm{\zeta}-\bm{1}))=\left(\sum_{i=0}^{d-1}(\mu_i p^n+\lambda_i n)p^{(d-i-1)n}\right)+\nu.
\end{equation}
By \eqref{CMeq2} and \eqref{CMeq3}, we must have
\begin{equation}\label{CMeq4}
\sum_{\bm{\zeta}\in W_n^d\setminus\{(1,\ldots,1)\}} v_p(F_0(\bm{\zeta}-\bm{1}))=\lambda(F)np^{(d-1)n}+\left(\sum_{i=1}^{d-1}(\mu_i p^n+\lambda_i n)p^{(d-i-1)n}\right)+\nu.
\end{equation}
By \eqref{CMeq1} and \eqref{CMeq4}, we obtain the assertion.
\end{proof}
\subsection{A proof of Iwasawa-type formula}\label{sec:Iwasawa_type}
We prove an Iwasawa-type formula for weighted graphs via the method of McGown--Valli\`{e}res \cite[\S 5-\S6]{MV24} and DuBose--Valli\`{e}res \cite[\S 5-\S6]{DV23}. Let $X$ be a weighted symmetric digraph. Here, we put $\Gamma:=\Z_p^d$ and $\Gamma_n:=\Gamma/\Gamma^{p^n}\cong(\Z\slash p^n\Z)^d$. Let $\alpha:\E(X)\to \Gamma$ be a voltage assignment. Then, for each $n\geq 0$, the map $\alpha_n:\E(X)\to \Gamma_n$ obtained by composing $\alpha$ and the $n$-th projection map $\Gamma\to\Gamma_n$ is also a voltage assignment. Let $X_n:=X(\Gamma_n, \alpha_n)$. By the definition of voltage assignments, the natural map $\Gamma_{n+1}\to\Gamma_n$ induces a cover $X_{n+1}\to X_n$ for each $n\geq 0$. Hence we have a tower of weighted symmetric digraphs
\[
X=X_0\leftarrow X_1\leftarrow X_2\leftarrow \cdots\leftarrow X_n\leftarrow \cdots,
\]
which is called a \emph{$\Z_p^d$-tower}.

For every $1\leq i,j\leq m$, we set  
\begin{align}
\E_{ij}:=\{e\in \E(X)\ |\ o(e)=v_i,t(e)=v_j\}.
\end{align}

\begin{defi}
Let $\{w_1,\ldots,w_r\}\subset\bar{\Q}_p$ be the image set of the weighted function $w$. From now on, let $K:=\Q_p(\zeta_p,w_1,\ldots,w_r)$. This $\zeta_p$ will be used in our proof of an analogue of Kida's formula (Theorem \ref{thm:weighted_kida_formula}). We use the notations defined in \S \ref{subsection:CM}. Define a map $\tau:\Z_p^d\to\Lambda_d$ by
\begin{align}
\tau(\bm{a}):=\tau(a_1,a_2,\ldots,a_d):=(1+T_1)^{a_1}(1+T_2)^{a_2}\cdots (1+T_d)^{a_d}.
\end{align}
Since $\Lambda_d$ is a topological ring, this definition makes sense. For the convergence, see \cite[\S 3]{MV24}, for example. We also put
\begin{align}
\bm{W}_{\tau,\alpha}:=\left(\sum_{e\in \E_{ij}}\tau(\alpha(e))w(e)\right)_{ij}\in M_m(\Lambda_{K,d}).
\end{align}
Here, we denote the set of $r\times r$-matrices with entries in a ring $R$ by $M_r(R)$. 
\end{defi}
\begin{defi}
We define the element $Q^W_{X,\alpha}(T)\in \Lambda_{K,d}$ by
\begin{align}
Q_{X,\alpha}^W(T):=Q_{X,\alpha}^W(T_1,T_2,\ldots,T_d):=\det(\bm{D}^W(X)-\bm{W}_{\tau,\alpha}).
\end{align}
Observing each entry of $W_{\tau,\alpha}$, we see that $Q_{X,\alpha}^W(T)$ satisfies the condition \eqref{semipowerseries}.
\end{defi}
We also put, for $\psi_n=(\psi_n)_{(1)}\otimes (\psi_n)_{(2)}\otimes\cdots\otimes(\psi_n)_{(d)}\in \widehat{\Gamma}_n$ $((\psi_n)_{(j)}\in \widehat{(\Z\slash p^n\Z)})$,
\begin{align}
\bm{\zeta_{\psi_n}}-\bm{1}:=((\psi_n)_{(1)}(1)-1,(\psi_n)_{(2)}(1)-1,\ldots,(\psi_n)_{(d)}(1)-1).
\end{align}
\begin{prop}
For $\psi_n=(\psi_n)_{(1)}\otimes (\psi_n)_{(2)}\otimes\cdots\otimes(\psi_n)_{(d)}\in \widehat{\Gamma}_n$, it holds that
\begin{align}\label{eq:Q_h}
Q_{X,\alpha}^W(\bm{\zeta_{\psi_n}}-\bm{1})=h_X(\psi_n,1,\alpha_n).
\end{align}
Moreover, if $\psi_n$ is a trivial character, then
\begin{align}
Q_{X,\alpha}^W(\bm{\zeta_{\psi_n}}-\bm{1})=h_X(\psi_n,1,\alpha_n)=0.
\end{align}
\end{prop}
\begin{proof}
For a character $\psi_n:\Gamma_n\to \overline{\mathbb{Q}}^*$ and an arbitrary element $\bm{a}=(a_1,\ldots,a_d)\in\Z_p^d$, it holds that
\begin{align}\label{rho_evaluate}
\tau(\bm{a})(\bm{\zeta_{\psi_n}}-\bm{1})=(\psi_n)_{(1)}(\bar{1})^{a_1}\cdots(\psi_n)_{(d)}(\bar{1})^{a_d}=(\psi_n)_{(1)}(\bar{a}_1)\cdots(\psi_n)_{(d)}(\bar{a}_d)=\psi_n(\bar{\bm{a}}).
\end{align}
Note that
\begin{align}
\bm{W}_{\tau,\alpha}(\bm{\zeta_{\psi_n}}-\bm{1})&=\left(\sum_{e\in \E_{ij}}\tau(\alpha(e))w(e)\right)_{ij}(\bm{\zeta_{\psi_n}}-\bm{1})\\
&\overset{\eqref{rho_evaluate}}{=}\left(\sum_{e\in \E_{ij}}\psi_n(\alpha_n(e))w(e)\right)_{ij}\\
&=\left(\sum_{\sigma\in \Gamma_n}(\psi_n(\sigma)\sum_{\substack{e\in \E_{ij}\\ \alpha_n(e)=\sigma}}w(e))\right)_{ij}\\
&=\sum_{\sigma\in \Gamma_n}\psi_n(\sigma)\bm{W}(\sigma)=\bm{W}_{\psi_n}\label{eq:key}.
\end{align}
Therefore we have
\begin{align}
Q^W_{X,\alpha}(\bm{\zeta_{\psi_n}}-\bm{1})&=\det(\bm{D}^W(X)-\bm{W}_{\tau,\alpha})(\bm{\zeta_{\psi_n}}-\bm{1})\\
&=\det(\bm{D}^W(X)-\bm{W}_{\psi_n})\\
&=h_X(\psi_n,1,\alpha_n).
\end{align}
If $\psi_n$ is a trivial character, then
\[
Q_{X,\alpha}^W(\bm{\zeta_{\psi_n}}-\bm{1})=\det(\bm{D}^W(X)-\bm{W}(X))=h_X(\psi_n,1,\alpha_n).
\]
Since it is well-known that $\det(\bm{D}^W(X)-\bm{W}(X))=0$, we have \begin{align}
Q_{X,\alpha}^W(\bm{\zeta_{\psi_n}}-\bm{1})=h_X(\psi_n,1,\alpha_n)=0.
\end{align}
\end{proof}
We prove a generalization of \cite[Theorem 6.2]{DV23} and \cite[Theorem 4.3]{KM24} to weighted graphs.
\begin{thm}[Iwasawa-type formula for weighted graphs]\label{thm:Iwasawa_Type}
Let $X$ be a weighted symmetric digraph. Let $\alpha:\mathbb{E}(X)\to\Gamma$ be a voltage assignment. For each $n\geq 0$, let $\alpha_n:\mathbb{E}(X)\to\Gamma_n$ be the voltage assignment obtained by composing $\alpha$ and $n$-th projection map $\Gamma\to\Gamma_n$. Let $X_n:=X(\Gamma_n,\alpha_n)$ and let $\kappa^W_n$ be the weighted complexity of $X_n$. We assume the following conditions.
\begin{itemize}
\item The weighted matrix $\bm{W}(X)$ is symmetric,
\item $X_n$ are connected for all $n\geq 0$.
\end{itemize}
Then we have the following.

(i) If there exists a non-negative integer $n_0$ such that $\kappa^W_{n_0}=0$, then $\kappa_n^W$ are $0$ for all $n\geq n_0$.

(ii) If $\kappa^W_n$ is not equal to $0$ for all $n\geq 0$, there exist $\mu_1,\ldots,\mu_{d-1}, \lambda_1,\ldots,\lambda_{d-1},\nu\in\Q$ such that, for all $n\gg 0$, it holds that
\begin{align}\label{eq:the_equality}
v_p(\kappa_n^W)=(\mu p^n+\lambda n)p^{(d-1)n}+\left(\sum_{
i=1}^{d-1}(\mu_ip^n+\lambda_in)p^{(d-i-1)n}\right)+\nu,
\end{align}
where $\mu:=\mu(Q^W)$ and 
\begin{align}
\lambda:=\begin{cases}
\lambda(Q^W)-1&(d=1),\\
\lambda(Q^W)&(d\geq 2).
\end{cases}
\end{align}
\end{thm}
$\mu(X,\alpha):=\mu$ is called \emph{the Iwasawa $\mu$-invariant of $(X,\alpha)$}, and $\lambda(X,\alpha):=\lambda$ is called \emph{the Iwasawa $\lambda$-invariant of $(X,\alpha)$}.
\begin{proof}
(i) It follows from Theorem \ref{thm:recurrence_formula}. 

(ii) Since $\kappa_n^W\neq 0$ for every $n\geq 0$, by Theorem \ref{thm:recurrence_formula}, we must have $h_X(\psi_n,1,\alpha_n)\neq 0$ for every $n\geq 0$ and $\psi_n\in\widehat{\Gamma}_n\setminus\{1_{\widehat{\Gamma}_n}\}$. Hence, by the equation \eqref{eq:Q_h}, $Q_{X,\alpha}^W(T)$ does not vanish for every ${\bm{\zeta}\in W_n^d\backslash \{(1,1,\ldots,1)\}}$. Applying Theorem \ref{thm:recurrence_formula} with $G=\Gamma_n$ for $n\geq0$, we have
\begin{align}
v_p(\kappa_n^W)&=-dn+v_p(\kappa^W_0)+\sum_{\psi_n\in \widehat{\Gamma}_n\backslash\{1_{\widehat{\Gamma}_n}\}}v_p(h_X(\psi_n,1,\alpha_n))\\
&\overset{\eqref{eq:Q_h}}{=}-dn+v_p(\kappa^W_0)+\sum_{\psi_n\in \widehat{\Gamma}_n\backslash\{1_{\widehat{\Gamma}_n}\}}v_p(Q_{X,\alpha}^W(\bm{\zeta_{\psi_n}}-\bm{1}))\\
&=-dn+v_p(\kappa^W_0)+\sum_{\bm{\zeta}\in W_n^d\backslash \{(1,1,\ldots,1)\}}v_p(Q_{X,\alpha}^W(\bm{\zeta}-\bm{1})).\label{lastline}\\
\end{align}
Applying Theorem \ref{thm:multi_class_number} with the equation \eqref{lastline}, we obtain the assertion.
\end{proof}
\begin{rem}
The gap between $\lambda(X,\alpha)$ and $\lambda(Q_{X,\alpha}^W)$ in the case $d=1$ derives from a graph theoretical analogue of the Iwasawa main conjecture. See \cite[Theorem 5.3 and Remark 5.4]{KM24}.
\end{rem}
\section{Weighted Kida's formula}\label{sec:Kida}
In this section, we prove weighted Kida's formula via the method of Ray--Valli\`{e}res \cite{RV22} with several modifications.
\begin{thm}[Kida's formula for weighted graphs]\label{thm:weighted_kida_formula}
Let $X$ be a weighted symmetric digraph. Let $\alpha:\mathbb{E}(X)\to\Gamma$ be a voltage assignment. For each $n\geq 0$, let $\alpha_n:\mathbb{E}(X)\to\Gamma_n$ be the voltage assignment obtained by composing $\alpha$ and $n$-th projection map $\Gamma\to\Gamma_n$. Let $X_n:=X(\Gamma_n,\alpha_n)$. Let $\pi:Y\to X$ be a Galois cover of degree $p$-power. We assume the following conditions.
\begin{enumerate}[label=(\arabic*)]
\item All $Y_n:=Y(\Gamma_n,\alpha_n\circ\pi_{\E})$ are connected, 
\item $\Im w_X\subset\mathcal{O}_K$,
\item $\kappa^W(Y_n)\neq 0$ for all $n\geq 0$.
\end{enumerate}
 Then we have $\mu(X,\alpha)=0$ if and only if $\mu(Y,\alpha\circ\pi_{\E})=0$. Moreover, if this equivalent condition is satisfied, then it holds
\begin{align} 
\lambda(Y,\alpha\circ\pi_{\E})=\begin{cases}
[Y:X]\lambda(X,\alpha)&(\text{if }d\geq 2),\\
[Y:X](\lambda(X,\alpha)+1)-1&(\text{if }d=1).
\end{cases}
\end{align}
\end{thm}
\begin{proof}
Note that all $X(\Gamma_n,\alpha_n)$ are also connected by the connectivity of $Y_n$. Also, since $\kappa^W(Y_n)\neq 0$ for every $n\geq 0$, by Theorem \ref{thm:recurrence_formula}, we must have $\kappa^W(X_n)\neq 0$ for every $n\geq0$. It suffices to show the theorem in the case $G:=\gal(Y\slash X)\cong \Z\slash p\Z$ by the same argument of \cite[Theorem 4.1]{RV22}.

For $Y/X$, there exists $\beta:\E_X\to G$ such that $Y=X(G,\beta)$. Define $\gamma_n:\E_X\ni e\mapsto (\beta(e),\alpha_n(e))\in G\times \Gamma_n$. Then, by the weighted-generalization of \cite[Proposition 3.8]{RV22}, we have $Y_n=Y(\Gamma_n,\alpha_n\circ\pi_{\E})\cong X(G\times \Gamma_n,\gamma_n)$ and $\gal(Y_n/X)\cong G\times\Gamma_n$.

For $\phi\in \widehat{G}$, we define the matrix $\bm{W}_{\phi,\tau,\alpha}\in M_m(\mathcal{O}\llbracket T_1,\cdots,T_ d\rrbracket)$ by
\begin{align}
\bm{W}_{\phi,\tau,\alpha}:=\left(\sum_{e\in \E_{ij}}\phi(\beta(e))\tau(\alpha(e))w(e)\right)_{ij}.
\end{align}
Put $Q^W_{\phi}(T):=\det(\bm{D}^W(X)-\bm{W}_{\phi,\tau,\alpha})\in\Lambda_{K,d}$. Here we note that elements of $\widehat{G\times \Gamma_n}$ are of the forms $\phi\otimes\psi_n$ ($\phi\in \hat{G}$, $\psi_n\in \widehat{\Gamma}_n$). As was the case with \cite[p.19, (4.2)]{RV22}, we have
\begin{align}
\indu_{1\times \Gamma_n}^{G\times \Gamma_n}(\psi_n)\cong\bigoplus_{\phi\in \hat{G}}\phi\otimes\psi_n.
\end{align}
Therefore, we have
\begin{align}
L(Y,t,\psi_n,\alpha_n\circ\pi_{\E})&\overset{\eqref{eq:Sato_Induction}}{=}L(X,t,\indu_{1\times \Gamma_n}^{G\times \Gamma_n}(\psi_n),\gamma_n)=L(X,t,\bigoplus_{\phi\in\hat{G}}\phi\otimes\psi_n,\gamma_n)\\
&\overset{\eqref{thm:sum_prod}}{=}\prod_{\phi\in \hat{G}}L(X,t,\phi\otimes\psi_n,\gamma_n).
\end{align}
By using the three-term determinant formula (Theorem \ref{thm:weighted_L_3term}) on each side, we have
\begin{align}
&(1-t^2)^{-\chi(Y)}\det\left(\bm{I}_{pm}-t\sum_{\sigma\in \Gamma_n}\bm{W}_{Y,\alpha_n\circ\pi_{\E}}(\sigma)\bigotimes\psi_n(\sigma)+t^2(\bm{D}^W(Y)-\bm{I}_{pm})\right)\\
&=\prod_{\phi\in\hat{G}}(1-t^2)^{-\chi(X)}\det\left(\bm{I}_m-t\sum_{\sigma\in G\times\Gamma_n}\bm{W}_{{X},\gamma_n}(\sigma)\bigotimes(\phi\otimes\psi_n)(\sigma)+t^2(\bm{D}^W(X)-\bm{I}_m)\right).
\end{align}
This implies
\begin{equation}\label{eqKida1}
h_Y(\psi_n,1,\alpha_n\circ\pi_{\E})=\prod_{\phi\in\widehat{G}}h_X(\phi\otimes\psi_n,1,\gamma_n).
\end{equation}
Since
\begin{align}
\bm{W}_{X,\phi\otimes\psi_n}&=\sum_{\sigma\in G\times\Gamma_n}(\phi\otimes\psi_n)(\sigma) \bm{W}_{X,\gamma_n}(\sigma)\\
&=\sum_{\sigma\in G\times\Gamma_n}((\phi\otimes\psi_n)(\sigma)\sum_{\substack{e\in \E_{X,ij}\\ \gamma_n(e)=\sigma}}w(e))\\
&=\sum_{e\in \E_{X,ij}}(\phi\otimes\psi_n)(\gamma_n(e))w(e)\\
&=\sum_{e\in \E_{X,ij}}\phi(\beta(e))\psi_n(\alpha_n(e))w(e),
\end{align}
we have
\begin{align}
Q^W_{\phi}(\bm{\zeta_{\psi_n}}-\bm{1})&\overset{\eqref{rho_evaluate}}{=}\det\left(\bm{D}^W(X)-\left(\sum_{e\in \E_{X,ij}}\phi(\beta(e))\psi_n(\alpha_n(e))w(e)\right)_{ij}\ \right)\\
&=\det(\bm{D}^W(X)-\bm{W}_{X,\phi\otimes \psi_n})=h_X(\phi\otimes\psi_n,1,\gamma_n)\label{eqKida2}.
\end{align}
Moreover, we have
\begin{align}
Q^W_{Y,\alpha\circ\pi_{\E}}(\bm{\zeta_{\psi_n}}-\bm{1})&=\det\left(\bm{D}^W(Y)-\left(\sum_{e\in \E_{Y,ij}}\tau(\alpha_n\circ\pi_{\E}(e))w(e)\right)_{ij}\ \right)(\bm{\zeta_{\psi_n}}-\bm{1})\\
&\overset{\eqref{rho_evaluate}}{=}\det\left(\bm{D}^W(Y)-\left(\sum_{e\in \E_{Y,ij}}\psi_n(\alpha_n\circ\pi_{\E}(e))w(e)\right)_{ij}\right)\\
&=\det(\bm{D}^W(Y)-\bm{W}_{Y,\psi_n,\alpha_n\circ\pi_{\E}})\\
&=h_Y(\psi_n,1,\alpha_n\circ\pi_{\E})\label{eqKida3}.
\end{align}
Hence, for arbitrary $\psi_n\in\widehat{\Gamma}_n$, we have
\begin{align}\label{eq:decomposition}
Q^W_{Y,\alpha\circ\pi_{\E}}(\bm{\zeta_{\psi_n}}-{1})\overset{\eqref{eqKida3}}{=}h_Y(\psi_n,1,\alpha_n\circ\pi_{\E})\overset{\eqref{eqKida1}}{=}\prod_{\phi\in\widehat{G}}h_X(\phi\otimes \psi_n,1,\gamma_n)\overset{\eqref{eqKida2}}{=}\prod_{\phi\in\widehat{G}}Q^W_{\phi}(\bm{\zeta_{\psi_n}}-\bm{1}).
\end{align}

We do not know whether $Q^W_{Y,\alpha\circ\pi_{\E}}(T)=\prod_{\phi\in\widehat{G}}Q^W_{\phi}(T)$ holds or not. However, by Theorem \ref{thm:multi_class_number}, we must have
\begin{align}
\mu(Q^W_{Y,\alpha\circ\pi_{\E}}(T))=\mu\left(\prod_{\phi\in \widehat{G}}Q^W_{\phi}(T)\right)=\sum_{\phi\in\widehat{G}}\mu(Q^W_{\phi}(T)),\\
\lambda(Q^W_{Y,\alpha\circ\pi_{\E}}(T))=\lambda\left(\prod_{\phi\in \widehat{G}}Q^W_{\phi}(T)\right)=\sum_{\phi\in\widehat{G}}\lambda(Q^W_{\phi}(T)).
\end{align}
Put $\mathcal{O}\llbracket T\rrbracket:=\mathcal{O}\llbracket T_1,\ldots, T_d\rrbracket$ and $\F_q\llbracket T\rrbracket:=\F_q\llbracket T_1,\ldots,T_d\rrbracket$. By the condition (2), we have $Q_{X,\alpha}^W,Q_{Y,\alpha\circ \pi_{\E}}^W\in \mathcal{O}\llbracket T\rrbracket$. Since $\zeta_p\in\mathcal{O}$ by the definition of $K$, we also have $Q^W_{\phi}\in \mathcal{O}\llbracket T\rrbracket$. Consider the reduction map
\begin{align}
\bar{\ \cdot\ }:\mathcal{O}\llbracket T\rrbracket\ni f(T)\mapsto \overline{f(T)}\in \mathcal{O}\llbracket T\rrbracket/\mathfrak{m}\llbracket T\rrbracket (\cong (\mathcal{O}/\mathfrak{m})\llbracket T\rrbracket\cong\F_q\llbracket T\rrbracket).
\end{align}
Since $\zeta_p\in\mathcal{O}$, we have $\zeta_p-1\in\mathfrak{m}$. Thus, we have $\bar{\zeta}_p=\bar{1}$ in $\mathcal{O}/\mathfrak{m}$, and so $\overline{\bm{W}}_{\phi,\tau,\alpha}=\overline{\bm{W}}_{\tau,\alpha}$ in $M_m(\F_q\llbracket T\rrbracket)$ for every $\phi\in\widehat{G}$. Hence
\begin{align}
\overline{Q^W_{\phi}(T)}=\overline{\det(\bm{D}^W(X)-\bm{W}_{\phi,\tau,\alpha})}=\overline{\det(\bm{D}^W(X)-\bm{W}_{\tau,\alpha})}=\overline{Q^W_{X,\alpha}(T)}\ \text{in}\ \F_q\llbracket T\rrbracket.
\end{align}
Therefore, we obtain
\begin{align}
\mu(X,\alpha)=0 &\Longleftrightarrow \mu(Q^W_{X,\alpha}(T))=0\\
&\Longleftrightarrow \overline{Q^W_{X,\alpha}(T)}\neq 0\ \text{in}\ \F_q\llbracket T\rrbracket\\
&\Longleftrightarrow \overline{Q^W_{\phi}(T)}\neq 0\ \text{in}\ \F_q\llbracket T\rrbracket,\ ^{\forall}\phi\in \widehat{G}\\
&\Longleftrightarrow \mu(Q^W_{\phi}(T))=0,\ ^{\forall}\phi\in \widehat{G}\\
&\Longleftrightarrow \mu(Q^W_{Y,\alpha\circ\pi_{\E}}(T))=\sum_{\phi\in\widehat{G}}\mu(Q^W_{\phi}(T))=0\\
&\Longleftrightarrow \mu(Y,\alpha\circ\pi_{\E})=0.
\end{align}
Now, suppose that the above equivalent condition is satisfied. If $d$ is greater than or equal to $2$, then we have
\begin{align}
\lambda(Y,\alpha\circ\pi_{\E})&=\lambda(Q^W_{Y,\alpha\circ\pi_{\E}}(T))=\sum_{\phi\in\widehat{G}}\lambda(Q^W_{\phi}(T))\\
&=\sum_{\phi\in\widehat{G}}\lambda(Q^W_{X,\alpha}(T))=[Y:X]\lambda(Q^W_{X,\alpha}(T))=[Y:X]\lambda(X,\alpha).
\end{align}
If $d$ is equal to $1$, then
\begin{align}
\lambda(Y,\alpha\circ\pi_{\E})+1=\lambda(Q^W_{Y,\alpha\circ\pi_{\E}}(T))=[Y:X]\lambda(Q^W_{X,\alpha}(T))=[Y:X](\lambda(X,\alpha)+1).
\end{align}
\end{proof}
\section{On calculation methods for characteristic elements}\label{sec:calculation}
Let $X$ be a weighted undirected graph with $V(X)=\{v_1,\cdots,v_m\}$ and $\alpha:\E(D_X)\to \Z_p^d$ be a voltage assignment. We denote the composition of the $n$-th projection $\Z_p^d\twoheadrightarrow (\Z\slash p^n\Z)^d$ and $\alpha$ by $\alpha_n$ for each $n\geq 1$. In this section, following \cite[\S 5-\S 6]{MV24}, we introduce an algorithm of calculation of $Q_{X,\alpha}^W(T):=Q_{D_X,\alpha}^W(T)$. We will utilize this algorithm in \S\ref{sec:examples}. To begin with, we define an \textit{orientation} of $\E(D_X)$. 
\begin{defi}
A subset $S_X$ of $\E(D_X)$ is called an \textit{orientation} if $S_X$ satisfies
\begin{itemize}
\item $S_X\cap \bar{S}_X=\emptyset$,
\item $\E(D_X)=S_X\cup\bar{S}_X$,
\end{itemize}
where $\bar{S}_X:=\{\bar{e}\in \E(D_X)\ |\ e\in S_X\}$.
\end{defi}
If $S_X$ is an orientation of $X$, then, by the definition of voltage assignments, for every (possibly infinite) group $G$, we find that giving a map $\alpha:S_X\to G$ is equivalent to giving a voltage assignment $\alpha:\E(D_X)\to G$. Thus we identify these two corresponding maps and call $\alpha:S_X \to G$ a voltage assignment as well.

We fix an orientation $S_X$ of $X$. We set, for $1\leq i,j\leq m$,
\begin{align}
S_{ij}:=\{s\in S_X\ |\ o(s)=v_i,\ t(s)=v_j\}.
\end{align}
The following theorem is the $\Z_p^d$-version of the weighted-analogue of \cite[Proposition 5.5]{MV24}.
\begin{defi}
Define square matrices $\bm{B}^W$ and $\bm{C}^W$ of order $m$ by
\begin{align}
B_{ij}^W:=\sum_{s\in S_{ij}}w(s)\mbox{ and }C_{ij}^W:=\sum_{s\in S_{ji}}w(s).
\end{align}
Note that, since $\bm{C}^W=(\bm{B}^W)^{\mathsf{T}}$, it suffices to calculate $\bm{B}^W$ in practical use.
\end{defi}
We prove the $\Z_p^d$-version of the weighted-analogue of \cite[Corollary 5.6]{MV24}.
\begin{thm}\label{thm:computation}
Write
\begin{align}
\alpha(S_X)=\{\bm{b}_1,\bm{b}_2,\cdots, \bm{b}_{N_{\alpha}}\}\subset \Z_p^d.
\end{align}
For each $1\leq k\leq N_{\alpha}$, put 
\begin{align}
\gamma_{ij}^{(k)}:=\sum_{\substack{s\in S_{ij}\\ \alpha(s)=\bm{b}_k}}w(s).
\end{align}
Define a square matrix $\bm{P}^W$ of order $m$ by
\begin{align}
&P_{ij}^W(X_1,X_2,\cdots,X_{N_{\alpha}},Y_1,Y_2,\cdots,Y_{N_{\alpha}})\\
&:=\gamma_{ij}^{(1)}X_1+\cdots+\gamma_{ij}^{(N_{\alpha})}X_{N_{\alpha}}+\gamma_{ji}^{(1)}Y_1+\cdots+\gamma_{ji}^{({N_{\alpha}})}Y_{N_{\alpha}},
\end{align}
and
\begin{align}
\bm{M}^W&:=\bm{D}^W(X)-\bm{B}^W-\bm{C}^W-\bm{P}^W\\
&=(d_{ij}^W-B_{ij}^W-C_{ji}^W-P_{ij}^W(X_1,\cdots,X_{N_{\alpha}},Y_1,\cdots,Y_{N_{\alpha}}))_{i,j}.
\end{align}
Put
\begin{align}
P^W(X_1,\cdots,X_{N_{\alpha}},Y_1,\cdots,Y_{N_{\alpha}}):=\det\bm{M}^W.
\end{align}
Then we have
\begin{align}
Q_{X,\alpha}^W(T_1,T_2,\ldots,T_d)=P^W(Q_{\bm{b}_1},Q_{\bm{b}_2},\ldots,Q_{\bm{b}_{N_{\alpha}}},Q_{-\bm{b}_1},Q_{-\bm{b}_2},\ldots,Q_{-\bm{b}_{N_{\alpha}}}),
\end{align}
where
\begin{align}
Q_{\bm{b}}:=Q_{\bm{b}}(T_1,\ldots,T_d):=\tau(\bm{b})-1\in \Lambda_d
\end{align} 
for $\bm{b}\in \Z_p^d.$
\end{thm}
\begin{proof}
It holds that
\begin{align}
&Q_{X,\alpha}^W(T_1,T_2,\ldots,T_d)=\det(\bm{D}^W(X)-\bm{W}_{\tau,\alpha})\\
&=\det\left(d_{ij}^W-\sum_{s\in S_{ij}}\tau(\alpha(s))w(s)-\sum_{s\in S_{ji}}\tau(-\alpha(s))w(s)\right)_{ij}\\
&=\det\left(d_{ij}^W-B_{ij}^W-C_{ij}^W-\sum_{s\in S_{ij}}(\tau(\alpha(s))-1)w(s)-\sum_{s\in S_{ji}}(\tau(-\alpha(s))-1)w(s)\right)_{ij}\\
&=\det\left(d_{ij}^W-B_{ij}^W-C_{ij}^W-\sum_{s\in S_{ij}}Q_{\alpha(s)}w(s)-\sum_{s\in S_{ji}}Q_{-\alpha(s)}w(s)\right)_{ij}\\
&=\det\left(d_{ij}^W-B_{ij}^W-C_{ij}^W-\sum_{k=1}^{N_{\alpha}} Q_{\bm{b}_k}\gamma_{ij}^{(k)}-\sum_{k=1}^{N_{\alpha}} Q_{-\bm{b}_k}\gamma_{ji}^{(k)}\right)_{ij}\\
&=\det\left(d_{ij}^W-B_{ij}^W-C_{ij}^W-P_{ij}^W(Q_{\bm{b}_1},Q_{\bm{b}_2},\ldots,Q_{\bm{b}_{N_{\alpha}}},Q_{-\bm{b}_1},Q_{-\bm{b}_2},\ldots,Q_{-\bm{b}_{N_{\alpha}}})\right)_{ij}\\
&=P^W(Q_{\bm{b}_1},Q_{\bm{b}_2},\ldots,Q_{\bm{b}_{N_{\alpha}}},Q_{-\bm{b}_1},Q_{-\bm{b}_2},\ldots,Q_{-\bm{b}_{N_{\alpha}}}).
\end{align}
\end{proof}
\begin{rem}
Since $\bm{M}^W(0,\ldots,0)$ coincides with the Laplacian matrix of $X$, we have $\det \bm{M}^W(0,\ldots,0)=0$. Hence we find that the constant term of $\det\bm{M}^W$ is $0$.
\end{rem}
\section{Examples}\label{sec:examples}
In this section, we give numerical examples of our main theorems. We calculate Iwasawa invariants for $\Z_p^d$-towers of weighted graphs via Theorem \ref{thm:computation}. For all examples in this section, we give figures of weighted undirected graphs, consider $\Z_p^d$-towers over the weighted symmetric digraphs derived from them, and observe how our main theorems work. Throughout this section, for simplicity, we assume that all weights are in $\Q_p$. 
\begin{ex}
\noindent Let $p=2$ and $d=1$. Let $X$ be a bouquet with two loops.
 \begin{equation*}
 \begin{tikzpicture}[baseline={([yshift=-1.7ex] current bounding box.center)}]
\node[draw=none,minimum size=2cm,regular polygon,regular polygon sides=1] (a) {};
\foreach \x in {1} 
  \fill (a.corner \x) circle[radius=0.7pt];
\draw (a.corner 1) to [in=50,out=130,loop] (a.corner 1);
\draw (a.corner 1) to [in=50,out=130,distance = 0.5cm,loop] (a.corner 1);
\end{tikzpicture}
 \end{equation*}
We assign a weight $a$ to the inner loop and a weight $b$ to the outer loop. Then $\bm{W}(X)= \begin{pmatrix} a+b  \\
\end{pmatrix}.$
Let $S=\{s_1,s_2\}$ be an orientation and let $\alpha:S\to \Z_2$ be a voltage assignment defined by $\alpha(s_1)=1$ and $\alpha(s_2)=1$.
Then we have 
\begin{align}
\bm{D}^W(X)= \begin{pmatrix}
2(a+b) \\
\end{pmatrix},\bm{B}^W= \begin{pmatrix}
a+b \\
\end{pmatrix},\bm{P}^W= \begin{pmatrix}
(a+b)X_1+(a+b)Y_1 \\
\end{pmatrix}
\end{align}
Thus, the matrix $\bm{M}^W$ is equal to $(-(a+b)X_1-(a+b)Y_1)$, and so we have $P^W(X_1,Y_1)= -(a+b)X_1-(a+b)Y_1$. By Theorem \ref{thm:computation}, we find that the formal power series $Q^W$ is $-(a+b)T^2(1+T)^{-1}$. If $a+b=0$, then $Q^W=0$. If $a+b\neq 0$, then we have $\mu=v_2(a+b)$ and $\lambda=1$.

In the case $a+b=0$, we find that $\kappa^W_X=1$ but $\kappa^W_n=0$ for $n\geq 1$. This means that even if the weighted complexity of the bottom layer is non-zero, those of the upper layers can be zero. Hence the non-zero assumption of $\kappa_n^W$ for $n\geq 0$ in Theorem \ref{thm:Iwasawa_Type} (ii) is consequential.

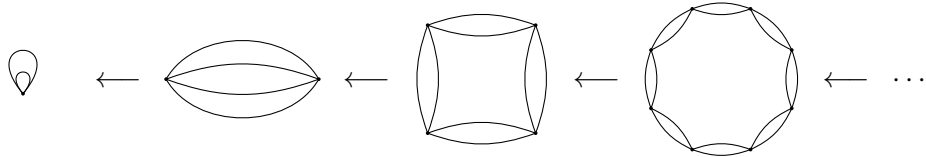
\begin{equation*}
	\begin{tikzpicture}[baseline={([yshift=-1.7ex] current bounding box.center)}]
\node[draw=none,minimum size=2cm,regular polygon,regular polygon sides=1] (a) {};
\foreach \x in {1}
  \fill (a.corner \x) circle[radius=0.7pt];
\draw (a.corner 1) to [in=50,out=130,loop] (a.corner 1);
\draw (a.corner 1) to [in=50,out=130,distance = 0.5cm,loop] (a.corner 1);
\end{tikzpicture}
\, \, \, \longleftarrow \, \, \,
\begin{tikzpicture}[baseline={([yshift=-0.6ex] current bounding box.center)}]
\node[draw=none,minimum size=2cm,regular polygon,rotate = -45,regular polygon sides=4] (a) {};

  \fill (a.corner 1) circle[radius=0.7pt];
  \fill (a.corner 3) circle[radius=0.7pt];
  
  \path (a.corner 1) edge [bend left=20] (a.corner 3);
  \path (a.corner 1) edge [bend left=60] (a.corner 3);
  \path (a.corner 1) edge [bend right=20] (a.corner 3);
  \path (a.corner 1) edge [bend right=60] (a.corner 3);

\end{tikzpicture}
\, \, \, \longleftarrow \, \, \,
  \begin{tikzpicture}[baseline={([yshift=-0.6ex] current bounding box.center)}]
\node[draw=none,minimum size=2cm,regular polygon,regular polygon sides=4] (a) {};
 \foreach \x in {1,2,3,4}
\fill (a.corner \x) circle[radius=0.7pt];

\foreach \y\z in {1/2,2/3,3/4,4/1}
  \path (a.corner \y) edge [bend left=20] (a.corner \z);

\foreach \y\z in {1/2,2/3,3/4,4/1}
  \path (a.corner \y) edge [bend right=20] (a.corner \z);

  \end{tikzpicture}
\, \, \, \longleftarrow \, \, \,
\begin{tikzpicture}[baseline={([yshift=-0.6ex] current bounding box.center)}]
\node[draw=none,minimum size=2cm,regular polygon,regular polygon sides=8] (a) {};
 \foreach \x in {1,2,...,8}
\fill (a.corner \x) circle[radius=0.7pt];

\foreach \y\z in {1/2,2/3,3/4,4/5,5/6,6/7,7/8,8/1}
  \path (a.corner \y) edge [bend left=20] (a.corner \z);

\foreach \y\z in {1/2,2/3,3/4,4/5,5/6,6/7,7/8,8/1}
  \path (a.corner \y) edge [bend right=20] (a.corner \z);

  \end{tikzpicture}
\, \, \, \longleftarrow \, \, \,\cdots
\end{equation*}
\end{ex}
\begin{ex}\label{ex:quantum}
\noindent Let $p=2$ and $d=1$. Let $X=\mathcal{K}_3$ be a complete graph with $3$ vertices. We assign the weights on graph $X$ by the following matrix
\begin{align}
\bm{W}(X)= \begin{pmatrix}
0 & b & a  \\
b & 0 & c \\
a & c & 0 \\
\end{pmatrix}.
\end{align}
Let $S=\{s_1,s_2,s_3\}$ be an orientation and $\alpha: S \to \Z_2 $ a voltage assignment defined by $\alpha(s_1)=0$, $\alpha(s_2)=0$ and $\alpha(s_3)=1$. Then we have 
 \begin{align}
 \bm{D}^{W}(X)= \begin{pmatrix}
a+b & 0 & 0  \\
0 & b+c & 0 \\
0 & 0 & c+a \\
\end{pmatrix}\!,\ \bm{B}^{W}= \begin{pmatrix}
0 & b & 0  \\
0 & 0 & c \\
a & 0 & 0 \\
\end{pmatrix}\!,\ 
\bm{P}^W= \begin{pmatrix}
0 & bX_1 & aY_1  \\
bY_1 & 0 & cX_2 \\
aX_1 & cY_2 & 0 \\
\end{pmatrix}
\end{align}
Thus, the matrix $\bm{M}^W$ is given by
\begin{align}
\begin{pmatrix}
a+b & -b(1+X_1) & -a(1+Y_1)  \\
-b(1+Y_1) & b+c & -c(1+X_2) \\
-a(1+X_1) & -c(1+Y_2) & c+a \\
\end{pmatrix},
\end{align}
and so we have 
\begin{align}
P^W= (a+b)(b+c)(c+a)-abc(1+X_1)^2(1+X_2)-abc(1+Y_1)^2(1+Y_2)\\ -c^2(a+b)(1+X_2)(1+Y_2)-a^2(b+c)(1+X_1)(1+Y_1)-b^2(c+a)(1+X_1)(1+Y_1). 
\end{align}
Therefore, the formal power series $Q^W$ is 
\begin{align}
2abc-abc(1+T)-abc(1+T)^{-1}=-abcT^2(1+T)^{-1}.
\end{align} 
Hence, if $abc=0$, then we have $Q^W=0$. On the other hand, if $abc\neq 0$, then we have $\mu=v_2(abc)$ and $\lambda=1$. 

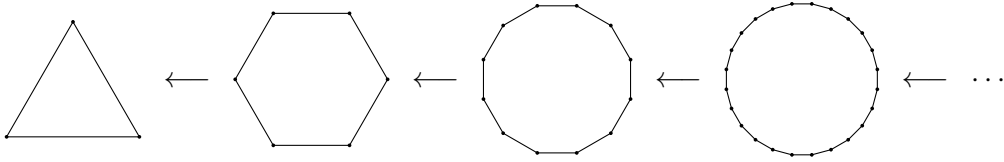
\begin{equation*}
\begin{tikzpicture}[baseline={([yshift=-0.6ex] current bounding box.center)}]

\node[draw=none,minimum size=2cm,regular polygon,regular polygon sides=3] (a) {};

\foreach \x in {1,2,...,3}
\fill (a.corner \x) circle[radius=0.7pt];

\foreach \y\z in {1/2,2/3,3/1}
  \path (a.corner \y) edge (a.corner \z);

\end{tikzpicture}
\, \, \, \longleftarrow \, \, \,
\begin{tikzpicture}[baseline={([yshift=-0.6ex] current bounding box.center)}]

\node[draw=none,minimum size=2cm,regular polygon,regular polygon sides=6] (a) {};

\foreach \x in {1,2,3,4,5,6}
\fill (a.corner \x) circle[radius=0.7pt];

\foreach \y\z in {1/2,2/3,3/4,4/5,5/6,6/1}
  \path (a.corner \y) edge (a.corner \z);

\end{tikzpicture}
\, \, \, \longleftarrow \, \, \,
\begin{tikzpicture}[baseline={([yshift=-0.6ex] current bounding box.center)}]

\node[draw=none,minimum size=2cm,regular polygon,regular polygon sides=12] (a) {};

\foreach \x in {1,2,3,4,5,6,7,8,9,10,11,12}
\fill (a.corner \x) circle[radius=0.7pt];

\foreach \y\z in {1/2,2/3,3/4,4/5,5/6,6/7,7/8,8/9,9/10,10/11,11/12,1/12}
  \path (a.corner \y) edge (a.corner \z);
  \end{tikzpicture}
 \, \, \, \longleftarrow \, \, \,
 \begin{tikzpicture}[baseline={([yshift=-0.6ex] current bounding box.center)}]

\node[draw=none,minimum size=2cm,regular polygon,regular polygon sides=24] (a) {};

\foreach \x in {1,2,3,4,5,6,7,8,9,10,11,12,13,14,15,16,17,18,19,20,21,22,23,24}
\fill (a.corner \x) circle[radius=0.7pt];

\foreach \y\z in {1/2,2/3,3/4,4/5,5/6,6/7,7/8,8/9,9/10,10/11,11/12,12/13,13/14,14/15,15/16,16/17,17/18,18/19,19/20,20/21,21/22,22/23,23/24,24/1}
  \path (a.corner \y) edge (a.corner \z);
  
  \end{tikzpicture}
  \, \, \, \longleftarrow \, \, \,\cdots
\end{equation*}
\end{ex}

\begin{ex}\label{ex3}
\noindent Let $p=2$ and $d=2$. Let $X$ be a bouquet with two loops. We assign a weight $a$ to the inner loop and a weight $b$ to the outer loop. The the weighted matrix of $X$ is $\bm{W}(X)= \begin{pmatrix} a+b \end{pmatrix}$.
Let $S=\{s_1,s_2 \}$ be an orientation and $\alpha:S\to \Z^2_2$ a voltage assignment defined by $\alpha(s_1)=(0,1)$ and $\alpha(s_2)=(1,0)$. Then we have $$\bm{D}^{W}(X)= \begin{pmatrix}
2(a+b) \\
\end{pmatrix}, \bm{B}^W= \begin{pmatrix}
a+b \\
\end{pmatrix}, \bm{P}^W= \begin{pmatrix}
aX_1+bX_2+aY_1+bY_2 \\
\end{pmatrix}. $$ Thus, the matrix $\bm{M}^W$ is equal to $$-(aX_1+bX_2+aY_1+bY_2),$$ and so we have $$P^W(X_1,X_2,Y_1,Y_2)=-(aX_1+bX_2+aY_1+bY_2).$$ If $a=b=0$, then the formal power series $Q^W$ is zero. Otherwise, $$Q^W=p^{\mu}(a'((1+T_1)-1 )+b'((1+T_2)-1 )+a'((1+T_1)^{-1}-1 )+b'((1+T_2)^{-1} -1 )), $$  where $\mu=\max\{v_2(a),v_2(b)\}$. 

Since $p=2$,
\[
(a',b')\mbox{ is either }(\bar{1},\bar{1}),(\bar{0},\bar{1}),\mbox{ or }(\bar{1},\bar{0}).
\]
If $(\bar{a'},\bar{b'})=(\bar{1},\bar{1})$, then
\begin{align}
  Q^W_0  &= \overline{((1+T_1)-1 )+((1+T_2)-1 )+((1+T_1)^{-1}-1 )+((1+T_2)^{-1} -1)} \\ &= \overline{T_1^{2}(1+T_1)^{-1}+T^2_2(1+T_2)^{-1}} \\ &\doteq \overline{T_1^2(1+T_2)+T_2^2(1+T_1)} \\ &= \overline{(T_1+T_2)^2+T_1T_2(T_1+T_2)} \\ &= \overline{(T_1+T_2)((1+T_1)(1+T_2)-1)} \\ &=
  \overline{((1+T_1)+(1+T_2))((1+T_1)(1+T_2)-1)} \\&=
  \overline{((1+T_1)-(1+T_2))((1+T_1)(1+T_2)-1)} \\&\doteq
  \overline{((1+T_1)(1+T_2)^{-1}-1)((1+T_1)(1+T_2)-1)}
\end{align}
in $\mathbb{F}_2\llbracket T_1,T_2 \rrbracket$. Here, $``\doteq"$ means $``$equal up to multiplication by units.$"$

If $(\bar{a'},\bar{b'})=(\bar{1},\bar{0})$, then $$Q^W_0=\overline{((1+T_1)-1)^2(1+T_1)^{-1}}.$$

If $(\bar{a'},\bar{b'})=(\bar{0},\bar{1})$, then $$Q^W_0=\overline{((1+T_2)-1)^2(1+T_2)^{-1}}.$$

By the definition of generalized $\mu$ and $\lambda$ invariants, we have $\mu=\max\{v_p(a),v_p(b)\}$ and  $\lambda=2$.
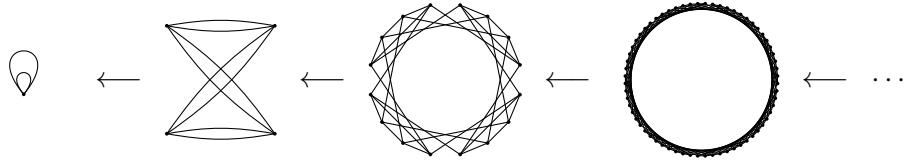
\begin{equation*}
	\begin{tikzpicture}[baseline={([yshift=-1.7ex] current bounding box.center)}]
\node[draw=none,minimum size=2cm,regular polygon,regular polygon sides=1] (a) {};
\foreach \x in {1}
  \fill (a.corner \x) circle[radius=0.7pt];
\draw (a.corner 1) to [in=50,out=130,loop] (a.corner 1);
\draw (a.corner 1) to [in=50,out=130,distance = 0.5cm,loop] (a.corner 1);
\end{tikzpicture}
\, \, \, \longleftarrow \, \, \,
\begin{tikzpicture}[baseline={([yshift=-0.6ex] current bounding box.center)}] 
\node[draw=none,minimum size=2cm,regular polygon,regular polygon sides=4] (a) {};

\foreach \x in {1,2,...,4}
  \fill (a.corner \x) circle[radius=0.7pt];

\path (a.corner 1) edge [bend left=10] (a.corner 2);
\path (a.corner 1) edge [bend right=10] (a.corner 2);

\path (a.corner 3) edge [bend left=10] (a.corner 4);
\path (a.corner 3) edge [bend right=10] (a.corner 4);

\path (a.corner 1) edge [bend left=10] (a.corner 3);
\path (a.corner 1) edge [bend right=10] (a.corner 3);

\path (a.corner 2) edge [bend left=10] (a.corner 4);
\path (a.corner 2) edge [bend right=10] (a.corner 4); 
\end{tikzpicture}
\, \, \, \longleftarrow \, \, \,
\begin{tikzpicture}[baseline={([yshift=-0.6ex] current bounding box.center)}] 
\node[draw=none,minimum size=2cm,regular polygon,rotate = 22.5,regular polygon sides=16] (a) {};

\foreach \x in {1,2,...,16}
  \fill (a.corner \x) circle[radius=0.7pt];

\foreach \y\z in {1/5,2/6,3/7,4/8,5/9,6/10,7/11,8/12,9/13,10/14,11/15,12/16,13/1,14/2,15/3,16/4}
  \path (a.corner \y) edge (a.corner \z);
  
\foreach \y\z in {1/2,2/3,3/4,5/6,6/7,7/8,9/10,10/11,11/12,13/14,14/15,15/16}
  \path (a.corner \y) edge (a.corner \z);
  
\foreach \y\z in {1/4,5/8,9/12,13/16}
  \path (a.corner \y) edge (a.corner \z);
\end{tikzpicture}
\, \, \, \longleftarrow \, \, \,\
\begin{tikzpicture}[baseline={([yshift=-0.6ex] current bounding box.center)}] 
\node[draw=none,minimum size=2cm,regular polygon,regular polygon sides=64] (a) {};

\foreach \x in {1,2,...,64}
  \fill (a.corner \x) circle[radius=0.7pt];

\foreach \y\z in {}
  \path (a.corner \y) edge (a.corner \z);
  
\foreach \y\z in {1/2,2/3,3/4,4/5,5/6,6/7,7/8}
  \path (a.corner \y) edge (a.corner \z);

\foreach \y\z in {9/10,10/11,11/12,12/13,13/14,14/15,15/16}
  \path (a.corner \y) edge (a.corner \z);

\foreach \y\z in {17/18,18/19,19/20,20/21,21/22,22/23,23/24}
  \path (a.corner \y) edge (a.corner \z);

\foreach \y\z in {25/26,26/27,27/28,28/29,29/30,30/31,31/32}
  \path (a.corner \y) edge (a.corner \z);
  
\foreach \y\z in {33/34,34/35,35/36,36/37,37/38,38/39,39/40}
  \path (a.corner \y) edge (a.corner \z);

\foreach \y\z in {41/42,42/43,43/44,44/45,45/46,46/47,47/48}
  \path (a.corner \y) edge (a.corner \z);

\foreach \y\z in {49/50,50/51,51/52,52/53,53/54,54/55,55/56}
  \path (a.corner \y) edge (a.corner \z);
  
\foreach \y\z in {57/58,58/59,59/60,60/61,61/62,62/63,63/64}
  \path (a.corner \y) edge (a.corner \z);

\foreach \y\z in {1/8,9/16,17/24,25/32,33/40,41/48,49/56,57/64}
  \path (a.corner \y) edge (a.corner \z);
  
\foreach \y\z in {1/9,2/10,3/11,4/12,5/13,6/14,7/15,8/16,9/17,10/18,11/19,12/20,13/21,14/22,15/23,16/24,17/25,18/26,19/27,20/28,21/29,22/30,23/31,24/32,25/33,26/34,27/35,28/36,29/37,30/38,31/39,32/40,33/41,34/42,35/43,36/44,37/45,38/46,39/47,40/48,41/49,42/50,43/51,44/52,45/53,46/54,47/55,48/56,49/57,50/58,51/59,52/60,53/61,54/62,55/63,56/64,57/1,58/2,59/3,60/4,61/5,62/6,63/7,64/8}
  \path (a.corner \y) edge (a.corner \z);
\end{tikzpicture}
\, \, \, \longleftarrow \, \, \,\cdots
\end{equation*}
\end{ex}

We also construct an example of Kida's formula (Theorem \ref{thm:weighted_kida_formula}) for non-trivial weighted graphs with $d>1$.
\begin{ex}\label{ex:kida}
Let $p=2$ and $d=2$. Let $X$ be a bouquet with four loops. We assign weights of $a_1$, $a_2$, $a_3$, and $a_4$, starting from the inner loop such that $a_1$ and $a_2$ are not equal to $0$. Then the weighted matrix is $\bm{W}(X)= \begin{pmatrix} a_1+a_2+a_3+a_4 
\end{pmatrix}$. Let $S=\{s_1,s_2,s_3,s_4\}$ be an orientation and $\alpha:S\to \Z^2_2$ a voltage assignment defined by $\alpha(s_1)=(0,1)$, $\alpha(s_2)=(1,0)$, $\alpha(s_3)=(0,0)$, and $\alpha(s_4)=(0,0)$. 

Then we have 
\begin{align}
\bm{D}^W(X)= \begin{pmatrix}
2(a_1+a_2+a_3+a_4) \\
\end{pmatrix}, \bm{B}_{X,\alpha}^W= \begin{pmatrix}
a_1+a_2+a_3+a_4 \\
\end{pmatrix},
\end{align}
and 
\begin{align}
\bm{P}_{X,\alpha}^W= \begin{pmatrix}
a_1X_1+a_2X_2+a_3X_3+a_4X_3+a_1Y_1+a_2Y_2+a_3Y_3+a_4Y_3 \\
\end{pmatrix}. 
\end{align}

Thus, the matrix $\bm{M}^W_{X,\alpha}$ is equal to $-(aX_1+a_2X_2+a_3X_3+a_4X_3+aY_1+a_2Y_2+a_3Y_3+a_4Y_3),$ and so we have $$P_{X,\alpha}^{W}=-( a_1X_1+a_2X_2+a_3X_3+a_4X_3+a_1Y_1+a_2Y_2+a_3Y_3+a_4Y_3) .$$ Therefore, since $a_1\neq 0$ and $a_2\neq 0$, we obtain $$Q_{X,\alpha}^W=2^{\mu}(a_1'((1+T_1)-1 )+a_2'((1+T_2)-1 )+a_1'((1+T_1)^{-1}-1 )+a_2'((1+T_2)^{-1} -1 )), $$  where $\mu(X,\alpha)=\max\{v_2(a_1),v_2(a_2)\}$. As was discussed in Example \ref{ex3}, we have $\lambda(X,\alpha)=2$.
 
Now let $\beta:S\to Q_8=\langle \sigma,\tau \rangle $ be a voltage assignment defined by $\beta(s_1)=1_{Q_8} ,\beta(s_2)=1_{Q_8}, \beta(s_3)=\sigma$ and $\beta(s_4)=\tau$, where $Q_8$ is the Quaternion group with identity $1_{Q_8}$. Then the covering $\pi:Y=X(Q_8,\beta)\to X$ is a Galois cover with Galois group isomorphic to $Q_8$.
	
Considering the voltage assignment $\alpha\circ\pi_{\E}:S_Y\to \mathbb{Z}_2^2$, we obtain another $\mathbb{Z}_2^{2}$-tower. Now we have $$ \bm{D}^W(Y)=2(a_1+a_2+a_3+a_4)\bm{I}_{8}$$ 

$$\bm{B}_{Y,\alpha\circ\pi_{\E}}^W= \begin{pmatrix}
a_1+a_2 & a_3 & 0 & 0 & a_4 & 0 & 0 & 0 \\

0 & a_1+a_2 & a_3 & 0 & 0 & a_4 & 0 & 0 \\

0 & 0 & a_1+a_2 & a_3 & 0 & 0 & a_4 & 0 \\ 

a_3 & 0 & 0 & a_1+a_2 & 0 & 0 & 0 & a_4 \\

0 & 0 & a_4 & 0 & a_1+a_2 & 0 & 0 & a_3 \\

0 & 0 & 0 & a_4 & a_3 & a_1+a_2 & 0 & 0 \\

a_4 & 0 & 0 & 0 & 0 & a_3 & a_1+a_2 & 0 \\

0 & a_4 & 0 & 0 & 0 & 0 & a_3 & a_1+a_2 \\

\end{pmatrix}, $$

$$\bm{P}_{Y,\alpha\circ\pi_{\E}}= \begin{pmatrix}
\mathcal{R} & a_3X_3 & 0 & a_3Y_3 & a_4X_3 & 0 & a_4Y_3 & 0 \\

a_3Y_3 & \mathcal{R} & a_3X_3 & 0 & 0 & a_4X_3 & 0 & a_4Y_3 \\

0 & a_3Y_3 & \mathcal{R} & a_3X_3 & a_4Y_3 & 0 & a_4X_3 & 0 \\

a_3X_3 & 0 & a_3Y_3 & \mathcal{R} & 0 & a_4Y_3 & 0 & a_4X_3 \\ 

a_4Y_4 & 0 & a_4X_3 & 0 & \mathcal{R} & a_3Y_3 & 0 & a_3X_3 \\

0 & a_4Y_3 & 0 & a_4X_3 & a_3X_3 & \mathcal{R} & a_3Y_3 & 0 \\

a_4X_3 & 0 & a_4Y_3 & 0 & 0 & a_3X_3 & \mathcal{R} & a_3Y_3 \\

0 & a_4X_3 & 0 & a_4Y_3 & a_3Y_3 & 0 & a_3X_3 & \mathcal{R} \\

\end{pmatrix},$$where we put $\mathcal{R}:=a_1X_1+a_2X_2+a_1Y_1+a_2Y_2$. By considering the situation where $\mu(X,\alpha)=\max\{v_2(a_1),v_2(a_2)\}=0$ and by using SageMath \cite{sage}, we can check that the polynomial $P_{Y,\alpha\circ\pi_{\E}}^W(X_1,X_2,X_3,Y_1,Y_2,Y_3)$ is equal to $$\overline{(a_1X_1+a_2X_2+a_3X_3+a_4X_3+a_1Y_1+a_2Y_2+a_3Y_3+a_4Y_3)}^8 \text{ in }\mathbb{F}_2[ X_1,X_2,X_3,Y_1,Y_2,Y_3].$$
Therefore, we have
\[
\overline{(Q_{Y,\alpha\circ\pi_{\E}}^W)}_0=\overline{((a_1((1+T_1)-1 )+a_2((1+T_2)-1 )+a_1((1+T_1)^{-1}-1)+a_2((1+T_2)^{-1} -1 )))}^8
\]
Accordingly, we obtain $\mu(Y,\alpha\circ\pi_{\E})=0$ and $\lambda(Y,\alpha\circ\pi_{\E})=16$.

Consequently, we have witnessed that Kida's formula (Theorem \ref{thm:weighted_kida_formula}) $$\lambda(Y,\alpha\circ\pi_{\E})=[Y:X]\lambda(X,\alpha)$$
actually holds for these covers.

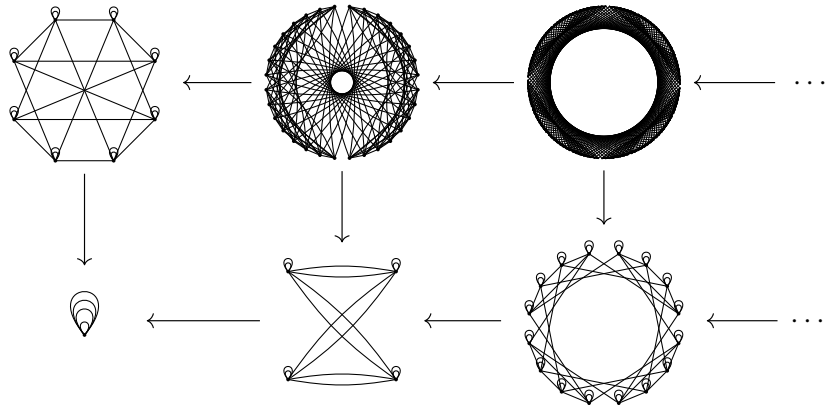
\begin{equation}
\begin{tikzcd}
 \begin{tikzpicture}[baseline={([yshift=-0.6ex] current bounding box.center)}] 
\node[draw=none,minimum size=2cm,regular polygon,regular polygon sides=8] (a) {};

\foreach \x in {1,2,...,8}
  \fill (a.corner \x) circle[radius=0.7pt];
  
\foreach \y\z in {1/2,1/5,1/7,1/8,2/3,2/4,2/6,8/4,8/3,8/6,3/5,3/7,7/4,7/6,6/5,4/5}
  \path (a.corner \y) edge (a.corner \z);
 
 \foreach \w in {1,2,...,8} 
\draw (a.corner \w) to [in=50,out=130,distance = 0.2cm,loop] (a.corner \w);

\foreach \w in {1,2,...,8} 
\draw (a.corner \w) to [in=50,out=130,distance = 0.3cm,loop] (a.corner \w);
 
\end{tikzpicture}\arrow[d]  & \begin{tikzpicture}[baseline={([yshift=-0.6ex] current bounding box.center)}] 
\node[draw=none,minimum size=2cm,regular polygon,rotate =11.25,regular polygon sides=32] (a) {};

\foreach \x in {1,2,...,32}
  \fill (a.corner \x) circle[radius=0.7pt];
  
\foreach \y\z in {1/2,2/3,3/4,4/1,5/6,6/7,7/8,8/5,9/10,10/11,11/12,12/9,13/14,14/15,15/16,16/13,17/18,18/19,19/20,20/17,21/22,22/23,23/24,24/21,25/26,26/27,27/28,28/25,29/30,30/31,31/32,32/29}
  \path (a.corner \y) edge (a.corner \z); 
  
\foreach \y\z in {1/5,2/6,3/7,4/8,9/13,10/14,11/15,12/16,17/21,18/22,19/23,20/24,25/29,26/30,27/31,28/32}
  \path (a.corner \y) edge (a.corner \z); 
   
\foreach \y\z in {1/7,2/8,3/5,4/6,9/15,10/16,11/13,12/14,17/23,18/24,19/21,20/22,25/31,26/32,27/29,28/30}
  \path (a.corner \y) edge (a.corner \z); 

\foreach \y\z in {1/9,2/10,3/11,4/12,5/13,6/14,7/15,8/16,17/25,18/26,19/27,20/28,21/29,22/30,23/31,24/32}
  \path (a.corner \y) edge [bend left=15] (a.corner \z);    

\foreach \y\z in {1/9,2/10,3/11,4/12,5/13,6/14,7/15,8/16,17/25,18/26,19/27,20/28,21/29,22/30,23/31,24/32}
  \path (a.corner \y) edge [bend right=15] (a.corner \z);      
  
\foreach \y\z in {1/17,2/18,3/19,4/20,5/21,6/22,7/23,8/24}
  \path (a.corner \y) edge [bend left=15] (a.corner \z);    

\foreach \y\z in {1/17,2/18,3/19,4/20,5/21,6/22,7/23,8/24}
  \path (a.corner \y) edge [bend right=15] (a.corner \z);      
  
\foreach \y\z in {9/25,10/26,11/27,12/28,13/29,14/30,15/31,16/32}
  \path (a.corner \y) edge [bend left=15] (a.corner \z);    

\foreach \y\z in {9/25,10/26,11/27,12/28,13/29,14/30,15/31,16/32}
  \path (a.corner \y) edge [bend right=15] (a.corner \z);        

\end{tikzpicture} \arrow[l]\arrow[d]  & \begin{tikzpicture}[baseline={([yshift=-0.6ex] current bounding box.center)}] 
\node[draw=none,minimum size=2cm,regular polygon,regular polygon sides=128] (a) {};

\foreach \x in {1,2,...,128}
  \fill (a.corner \x) circle[radius=0.3pt];
  
  \foreach \y\z in {1/9,2/10,3/11,4/12,5/13,6/14,7/15,8/16,9/17,10/18,11/19,12/20,13/21,14/22,15/23,16/24,17/25,18/26,19/27,20/28,21/29,22/30,23/31,24/32,25/1,26/2,27/3,28/4,29/5,30/6,31/7,32/8}
  \path (a.corner \y) edge (a.corner \z); 

\foreach \y\z in {33/41,34/42,35/43,36/44,37/45,38/46,39/47,40/48,41/49,42/50,43/51,44/52,45/53,46/54,47/55,48/56,49/57,50/58,51/59,52/60,53/61,54/62,55/63,56/64,57/33,58/34,59/35,60/36,61/37,62/38,63/39,64/40}
  \path (a.corner \y) edge (a.corner \z); 

\foreach \y\z in {65/73,66/74,67/75,68/76,69/77,70/78,71/79,72/80,73/81,74/82,75/83,76/84,77/85,78/86,79/87,80/88,81/89,82/90,83/91,84/92,85/93,86/94,87/95,88/96,89/65,90/66,91/67,92/68,93/69,94/70,95/71,96/72}
  \path (a.corner \y) edge (a.corner \z);   
  
  \foreach \y\z in {97/105,98/106,99/107,100/108,101/109,102/110,103/111,104/112,105/113,106/114,107/115,108/116,109/117,110/118,111/119,112/120,113/121,114/122,115/123,116/124,117/125,118/126,119/127,120/128,121/97,122/98,123/99,124/100,125/101,126/102,127/103,128/104}
  \path (a.corner \y) edge (a.corner \z);   
\foreach \x in {1,2,...,128}
  \fill (a.corner \x) circle[radius=0.3pt];
  
  \foreach \y\z in {1/33,2/34,3/35,4/36,5/37,6/38,7/39,8/40}
  \path (a.corner \y) edge (a.corner \z);
  
  \foreach \y\z in {9/41, 10/42, 11/43, 12/44, 13/45, 14/46, 15/47, 16/48}
  \path (a.corner \y) edge (a.corner \z);
  
  \foreach \y\z in {17/49, 18/50, 19/51, 20/52, 21/53, 22/54, 23/55, 24/56}
  \path (a.corner \y) edge (a.corner \z);
  
  \foreach \y\z in {25/57, 26/58, 27/59, 28/60, 29/61, 30/62, 31/63, 32/64}
  \path (a.corner \y) edge (a.corner \z);
  
  \foreach \y\z in {33/65, 34/66, 35/67, 36/68, 37/69, 38/70, 39/71, 40/72}
  \path (a.corner \y) edge (a.corner \z);
  
  \foreach \y\z in {41/73, 42/74, 43/75, 44/76, 45/77, 46/78, 47/79, 48/80}
  \path (a.corner \y) edge (a.corner \z);
  
  \foreach \y\z in {49/81, 50/82, 51/83, 52/84, 53/85, 54/86, 55/87, 56/88}
  \path (a.corner \y) edge (a.corner \z);
  
  \foreach \y\z in {57/89, 58/90, 59/91, 60/92, 61/93, 62/94, 63/95, 64/96}
  \path (a.corner \y) edge (a.corner \z);
  
  \foreach \y\z in {65/97, 66/98, 67/99, 68/100, 69/101, 70/102, 71/103,72/104}
  \path (a.corner \y) edge (a.corner \z);
  
  \foreach \y\z in {73/105, 74/106, 75/107, 76/108, 77/109, 78/110, 79/111, 80/112}
  \path (a.corner \y) edge (a.corner \z);
  
  \foreach \y\z in {81/113, 82/114, 83/115, 84/116, 85/117, 86/118, 87/119, 88/120}
  \path (a.corner \y) edge (a.corner \z);
  
  \foreach \y\z in {89/121, 90/122, 91/123, 92/124, 93/125, 94/126, 95/127, 96/128}
  \path (a.corner \y) edge (a.corner \z);
  
  \foreach \y\z in {97/1,98/2,99/3,100/4,101/5,102/6,103/7,104/8}
  \path (a.corner \y) edge (a.corner \z);
  
  \foreach \y\z in {105/9, 106/10, 107/11, 108/12, 109/13, 110/14, 111/15, 112/16}
  \path (a.corner \y) edge (a.corner \z);
  
  \foreach \y\z in {113/17, 114/18, 115/19, 116/20, 117/21, 118/22, 119/23, 120/24}
  \path (a.corner \y) edge (a.corner \z);
  
  \foreach \y\z in {121/25, 122/26, 123/27, 124/28, 125/29, 126/30, 127/31, 128/32}
  \path (a.corner \y) edge (a.corner \z);
\foreach \y\z in {1/2,2/3,3/4,4/1,5/6,6/7,7/8,8/5,9/10,10/11, 11/12, 12/9,13/14, 14/15, 15/16, 16/13,17/18, 18/19, 19/20, 20/17,21/22, 22/23, 23/24, 24/21,25/26, 26/27, 27/28, 28/25,29/30, 30/31, 31/32, 32/29,33/34, 34/35, 35/36, 36/33,37/38, 38/39, 39/40, 40/37,41/42, 42/43, 43/44, 44/41,45/46, 46/47, 47/48, 48/45,49/50, 50/51, 51/52, 52/49,53/54, 54/55, 55/56, 56/53,57/58, 58/59, 59/60, 60/57,61/62, 62/63, 63/64, 64/61,65/66, 66/67, 67/68, 68/65,69/70, 70/71, 71/72, 72/69,73/74, 74/75, 75/76, 76/73,77/78, 78/79, 79/80, 80/77,81/82, 82/83, 83/84, 84/81,85/86, 86/87, 87/88, 88/85,89/90, 90/91, 91/92, 92/89,93/94, 94/95, 95/96, 96/93,97/98, 98/99, 99/100, 100/97,101/102, 102/103, 103/104, 104/101,105/106, 106/107, 107/108, 108/105,109/110, 110/111, 111/112, 112/109,113/114, 114/115, 115/116, 116/113,117/118, 118/119, 119/120, 120/117,121/122, 122/123, 123/124, 124/121,125/126, 126/127, 127/128, 128/125}
  \path (a.corner \y) edge (a.corner \z); 
\foreach \y\z in {1/5,2/6,3/7,4/8,1/7,2/8,3/5,4/6, 9/13, 10/14, 11/15, 12/16, 9/15, 10/16, 11/13, 12/14,17/21, 18/22, 19/23, 20/24, 17/23, 18/24, 19/21, 20/22,25/29, 26/30, 27/31, 28/32, 25/31, 26/32, 27/29, 28/30,33/37, 34/38, 35/39, 36/40, 33/39, 34/40, 35/37, 36/38,41/45, 42/46, 43/47, 44/48, 41/47, 42/48, 43/45, 44/46,49/53, 50/54, 51/55, 52/56, 49/55, 50/56, 51/53, 52/54,57/61, 58/62, 59/63, 60/64, 57/63, 58/64, 59/61, 60/62,65/69, 66/70, 67/71, 68/72, 65/71, 66/72, 67/69, 68/70,73/77, 74/78, 75/79, 76/80, 73/79, 74/80, 75/77, 76/78,81/85, 82/86, 83/87, 84/88, 81/87, 82/88, 83/85, 84/86,89/93, 90/94, 91/95, 92/96, 89/95, 90/96, 91/93, 92/94,97/101, 98/102, 99/103, 100/104, 97/103, 98/104, 99/101, 100/102,105/109, 106/110, 107/111, 108/112, 105/111, 106/112, 107/109, 108/110,113/117, 114/118, 115/119, 116/120, 113/119, 114/120, 115/117, 116/118,121/125, 122/126, 123/127, 124/128, 121/127, 122/128, 123/125, 124/126}
  \path (a.corner \y) edge (a.corner \z); 
\end{tikzpicture}  \arrow[d]\arrow[l]  & {}\cdots \arrow[l]   \\  
  \begin{tikzpicture}[baseline={([yshift=-1.7ex] current bounding box.center)}]
\node[draw=none,minimum size=2cm,regular polygon,regular polygon sides=1] (a) {};
\foreach \x in {1}
  \fill (a.corner \x) circle[radius=0.7pt];
\draw (a.corner 1) to [in=50,out=130,distance = 1.2cm,loop] (a.corner 1);
\draw (a.corner 1) to [in=50,out=130,distance = 0.8cm,loop] (a.corner 1);
\draw (a.corner 1) to [in=50,out=130,distance = 0.6cm,loop] (a.corner 1);
\draw (a.corner 1) to [in=50,out=130,distance = 0.3cm,loop] (a.corner 1);
\end{tikzpicture}  &  \begin{tikzpicture}[baseline={([yshift=-0.6ex] current bounding box.center)}] 
\node[draw=none,minimum size=2cm,regular polygon,regular polygon sides=4] (a) {};

\foreach \x in {1,2,...,4}
  \fill (a.corner \x) circle[radius=0.7pt];

\path (a.corner 1) edge [bend left=10] (a.corner 2);
\path (a.corner 1) edge [bend right=10] (a.corner 2);

\path (a.corner 3) edge [bend left=10] (a.corner 4);
\path (a.corner 3) edge [bend right=10] (a.corner 4);

\path (a.corner 1) edge [bend left=10] (a.corner 3);
\path (a.corner 1) edge [bend right=10] (a.corner 3);

\path (a.corner 2) edge [bend left=10] (a.corner 4);
\path (a.corner 2) edge [bend right=10] (a.corner 4); 

\foreach \w in {1,2,...,4} 
\draw (a.corner \w) to [in=50,out=130,distance = 0.2cm,loop] (a.corner \w);

\foreach \w in {1,2,...,4} 
\draw (a.corner \w) to [in=50,out=130,distance = 0.3cm,loop] (a.corner \w);

\end{tikzpicture}\arrow[l] &  \begin{tikzpicture}[baseline={([yshift=-0.6ex] current bounding box.center)}] 
\node[draw=none,minimum size=2cm,regular polygon,rotate = 22.5,regular polygon sides=16] (a) {};

\foreach \x in {1,2,...,16}
  \fill (a.corner \x) circle[radius=0.7pt];

\foreach \y\z in {1/5,2/6,3/7,4/8,5/9,6/10,7/11,8/12,9/13,10/14,11/15,12/16,13/1,14/2,15/3,16/4}
  \path (a.corner \y) edge (a.corner \z);
  
\foreach \y\z in {1/2,2/3,3/4,5/6,6/7,7/8,9/10,10/11,11/12,13/14,14/15,15/16}
  \path (a.corner \y) edge (a.corner \z);
  
\foreach \y\z in {1/4,5/8,9/12,13/16}
  \path (a.corner \y) edge (a.corner \z);
  
\foreach \w in {1,2,...,16} 
\draw (a.corner \w) to [in=50,out=130,distance = 0.2cm,loop] (a.corner \w);

\foreach \w in {1,2,...,16} 
\draw (a.corner \w) to [in=50,out=130,distance = 0.3cm,loop] (a.corner \w);

\end{tikzpicture} \arrow[l]  &  \cdots\arrow[l]  {}  \\     
\end{tikzcd}
\end{equation}
\end{ex}
\begin{rem}
We can also construct examples of Kida's formula when $G$ is the dihedral group $D_4$ of order $8$ by discussing exactly the same way as Example \ref{ex:kida}. It is still wrapped in mystery for us whether there exists a $\Z_p^d$-tower with non-zero $Q^W(T_1,\ldots,T_d)$ such that $Q^W(\bm{\zeta}-\bm{1})=0$ for some $n\geq0$ and $\bm{\zeta}\in W_n^d$.
\end{rem}

\section{Application to quantum walks}\label{sec:quantum_walks}
In this section, To construct a first step of application of the ideas of Iwasawa theory combined with the weighted graph theory, we establish an approximation theorem on the values of the characteristic polynomials of the transition matrices of discrete-time quantum walks in a tower of graphs.

Let $X$ be a symmetric digraph with vertices $V(X)=\{v_1,\ldots,v_m\}$ and set of directed edges $\mathbb{E}(X)=\{e_1,\ldots,e_l,e_{l+1}=\bar{e}_1,\ldots,e_{2l}=\bar{e}_l\}$. For a vertex $v$ of $X$, put $d_v:=\#\E_{X,v}$.
\begin{defi}
For each $1\leq i,j\leq 2l$, put
\[
u_{ij}:=\begin{cases}
2/d_{o(e_i)}(=2/d_{t(e_j)})&\mbox{if }o(e_i)=t(e_j),\ \bar{e}_i\neq e_j,\\
2/d_{t(e_i)}-1&\mbox{if }\bar{e}_i=e_j,\\
0&\mbox{otherwise}.
\end{cases}
\]
The matrix $\bm{U}:=(u_{ij})_{ij}$ is called the \emph{transition matrix in a discrete-time quantum walk in $X$}. 
\end{defi}
The transition matrices of discrete-time quantum walks in graphs are closely related to the Ihara zeta functions of weighted graphs. Equip a weight $w:\mathbb{E}(X)\ni e\mapsto 2/d_{o(e)}\in\Q$ with $X$. Then we have
\[
\zeta(X, t)^{-1}=\det(\bm{I}_{2l}-t\bm{U}).
\]
Making use of this equality, Konno and Sato proved the celebrated Konno--Sato theorem, which has the following application to calculate the eigenvalues of $\bm{U}$.

\begin{thm}[{\cite[Proposition 2.2]{EHSW06}}, {\cite[Corollary 4.2]{KS12}}]\label{corofKS}
$\bm{U}$ has $2l$ eigenvalues of the forms
\[
a=\frac{1}{2}a_{\bm{W}}\pm\frac{1}{2}i\sqrt{4-a_{\bm{W}}^2}\in\bar{\Q},
\]
where $a_{\bm{W}}$ is an eigenvalue of $\bm{W}(X)$. The remaining $2(l-m)$ eigenvalues of $\bm{U}$ are $\pm 1$ with equal multiplicities.
\end{thm}
By Theorem \ref{corofKS}, one can calculate the eigenvalues of $\bm{U}$. While Theorem \ref{corofKS} is about the eigenvalues of $\bm{U}$, our following Iwasawa theoretical theorem is about the values of the characteristic polynomial $\det(u\bm{I}_{2l_n}-\bm{U})$ when a non-eigenvalue $a\in\bar{\Q}_p$ is substituted (recall that we have been fixing an embedding $\bar{\Q}\hookrightarrow\bar{\Q}_p$). This can be regarded as a graph discrete-time quantum walk analogue of Greenberg's class number conjecture (Conjecture \ref{conj:Greenberg}). Note that our weight function is compatible with the locally isomorphism part of the definition of (unramified) covers of weighted symmetric digraphs (Definition \ref{def:cover} (2)).

\begin{thm}\label{thm:quantumwalk}
Let $X$ be a symmetric digraph. Let $\alpha:\mathbb{E}(X)\to\Gamma$ be a voltage assignment. For each $n\geq 0$, let $\alpha_n:\mathbb{E}(X)\to\Gamma_n$ be the voltage assignment obtained by composing $\alpha$ and $n$-th projection map $\Gamma\to\Gamma_n$. Let $X_n:=X(\Gamma_n,\alpha_n)$ and let $2l_n$ denote the number of edges of $X_n$. Let $\bm{U}_n$ be the transition matrix of $X_n$. Let $a$ be an element of $\bar{\Q}_p$ such that, for every $n\geq 0$, $a$ is not an eigenvalue of $\bm{U}_n$. Let $K:=\Q_p(a)$ and define
\[
Q_a^W(T_1,\ldots,T_d):=\det(a^2\bm{I}_m-a\bm{W}_{\tau,\alpha}+(\bm{D}^W(X)-\bm{I}_m))\in\Lambda_{K,d}.
\]
Then there exist $\mu_1,\ldots,\mu_{d-1},\lambda_1,l\ldots,\lambda_{d-1},\nu\in\Q$ such that, for all $n\gg0$,
\begin{equation}\label{eq:quantumwalk}
v_p(\det(a\bm{I}_{2l_n}-\bm{U}_n))=(\mu p^n+\lambda n)p^{(d-1)n}+\left(\sum_{i=1}^{d-1}(\mu_ip^n+\lambda_in)p^{(d-i-1)n}\right)+\nu,
\end{equation}
where $\mu=\mu(Q_a^W)-\chi(X)v_p(a^2-1)$ and $\lambda=\lambda(Q_a^W)$.
\end{thm}
\begin{proof}
By the proof of \cite[Theorem 4.1]{KS12}, we have
\[
\zeta(X_n, t)^{-1}=\det(\bm{I}_{2l_n}-t\bm{U}_n).
\]
By \cite[Corollary 1]{MS04}(or for non-simple cases \cite[Theorem 6]{Sat07}), we have
\[
\zeta(X_n,t)^{-1}=(1-t^2)^{-p^{dn}\chi(X)}\prod_{{\psi_n}\in\widehat{\Gamma}_n}h_X({\psi_n},t,\alpha_n).
\]
Therefore, we have
\[
\det(\bm{I}_{2l_n}-t\bm{U}_n)=(1-t^2)^{-p^{dn}\chi(X)}\prod_{{\psi_n}\in\widehat{\Gamma}_n}h_X({\psi_n},t,\alpha_n).
\]
Put $u:=\frac{1}{t}$. Then, by an easy calculation, we have
\[
\det(u\bm{I}_{2l_n}-\bm{U}_n)=(u^2-1)^{-p^{dn}\chi(X)}\prod_{{\psi_n}\in\widehat{\Gamma}_n}\det(u^2\bm{I}_m-u\bm{W}_{\psi_n}+(\bm{D}^W(X)-\bm{I}_m)).
\]
Since $a$ is not an eigenvalue of $\bm{U}_n$, we have $\det(u\bm{I}_{2l_n}-\bm{U}_n)\neq 0$. Therefore, we obtain
\begin{align}
&v_p(\det(a\bm{I}_{2l_n}-\bm{U}_n))\\
&=-p^{dn}\chi(X)v_p(a^2-1)+\sum_{{\psi_n}\in\widehat{\Gamma}_n}v_p(\det(a^2\bm{I}_m-a\bm{W}_{{\psi_n}}+(\bm{D}^W(X)-\bm{I}_m)))\\
&\overset{\eqref{eq:key}}{=}-p^{dn}\chi(X)v_p(a^2-1)+\sum_{{\psi_n}\in\widehat{\Gamma}_n}v_p(\det(a^2\bm{I}_m-a\bm{W}_{\tau,\alpha}(\bm{\zeta_{{\psi_n}}}-\bm{1})+(\bm{D}^W(X)-\bm{I}_m)))\\
&=-p^{dn}\chi(X)v_p(a^2-1)+\sum_{{\psi_n}\in\widehat{\Gamma}_n}v_p(Q_a^W(\bm{\zeta_{{\psi_n}}}-\bm{1}))\\
&=-p^{dn}\chi(X)v_p(a^2-1)+v_p(Q_a^W(\bm{0}))+\sum_{\bm{\zeta}\in W_n^d\backslash\{(1,1,\ldots,1)\}}v_p(Q_a^W(\bm{\zeta}-\bm{1})).
\end{align}
By Theorem \ref{thm:multi_class_number}, we complete the proof.
\end{proof}
\begin{rem}\label{rem:quantumwalk}
There exists a calculation method for $``$how large $n$ should be to make the equation \eqref{eq:quantumwalk} hold$"$ in the case $d=1$. In fact, the equation \eqref{eq:quantumwalk} holds for every $n\geq 0$ satisfying $\varphi(p^n)\geq \lambda(Q^W_a)$, where $\varphi$ stands for the Euler's totient function. See the proofs of \cite[Theorem 6.1]{MV24} and \cite[Theorem 4.1]{MV23}.
\end{rem}
Finally, we give an example of Theorem $\ref{thm:quantumwalk}$.
\begin{ex}
Consider the situation of Example \ref{ex:quantum}. In this situation, we have $\chi(X)$=0. Let $a\in\Q_2$. Then we have
\[
a\bm{I}_{3\cdot 2^n}-\bm{W}(X_n)=
\begin{pmatrix}
a&-1&0&\cdots&\cdots&0&-1\\
-1&a&-1&0&\cdots&\cdots&0\\
0&-1&a&-1&0&\cdots&0\\
\vdots&\vdots&\vdots&\ddots&&\vdots&\vdots\\
0&\cdots&0&-1&a&-1&0\\
0&\cdots&\cdots&0&-1&a&-1\\
-1&0&\cdots&\cdots&0&-1&a
\end{pmatrix}.
\]
By the theory of circulant matrices, we have
\[
\spec \bm{W}(X_n)=\{2\cos\frac{2k\pi}{3\cdot 2^n}\ |\ 0\leq k\leq 3\cdot 2^n-1\}.
\]
By Theorem \ref{corofKS}, for every $n\geq 0$, we have
\begin{eqnarray*}
\spec \bm{U}_n&=&\{\pm 1,\cos\frac{2k\pi}{3\cdot 2^n}\pm i\sin\frac{2k\pi }{3\cdot 2^n}\ |\ 0\leq k\leq 3\cdot 2^n-1\}\\
&=&\{\pm 1, \zeta_{3\cdot 2^n}^k\ |\ 0\leq k\leq 3\cdot 2^n-1\}.
\end{eqnarray*}
Let $a\in\Q_2\setminus(\bigcup_{n\geq 0}\spec \bm{U}_n)$. Then
\begin{eqnarray*}
Q_a^W(T)&=&\det(a^2\bm{I}_3-a\bm{W}_{\tau,\alpha}(X)+(\bm{D}^W(X)-\bm{I}_3))\\
&=&\det
\begin{pmatrix}
a^2+1&-a&-(1+T)^{-1}a\\
-a&a^2+1&-a\\
-(1+T)a&-a&a^2+1
\end{pmatrix}\\
&=&(a^3-1)^2-a^3T^2(1+T)^{-1}.
\end{eqnarray*}
If $v_2(a^3-1)=0$, then $a^3\in\Z_2$ and $Q^W_a(T)$ is a unit in $\Lambda_1$. Therefore, $\mu=0$ and $\lambda=0$. If $v_2(a^3-1)\geq 1$, then we have $a^3\in\Z_2$ and $v_2(a^3)=0$. Hence $\mu=0$ and $\lambda=2$. If $v_2(a^3-1)\leq -1$, then we have $2^{-2v_2(a^3-1)}Q_a^W(T)\in\Lambda_1$ and $2^{-2v_2(a^3-1)+1}Q^W_a(T)\notin\Lambda_1$. Thus $\mu=2v_2(a^3-1)$ and $\lambda=0$.

We shall witness how Theorem \ref{thm:quantumwalk} actually works more concretely. By the definition of the transition matrices $\bm{U}_n$, a simple calculation allows us to obtain
\[
\det(a\bm{I}_{3\cdot 2^n}-\bm{U}_n)=(a^{3\cdot 2^n}-1)^2.
\]

If $a=2$, then $v_2(2^3-1)=0$. In this case, for every $n\geq 0$, we obviously have 
\[
v_2(\det(2\bm{I}_{3\cdot 2^n}-\bm{U}_n))=v_2((2^{3\cdot 2^n}-1)^2)=0=0\cdot 2^n+0\cdot n+0.
\]

If $a=3$, then $v_2(3^3-1)\geq 1$. In this case, by using SageMath \cite{sage}, we have
\begin{align}
(3^3-1)^2& = 2^2\cdot13^2,\\
(3^6-1)^2&= 2^6\cdot7^2\cdot13^2, \\
(3^{12}-1)^2&= 2^8\cdot5^2\cdot7^2\cdot13^2\cdot73^2, \\
(3^{24}-1)^2&= 2^{10}\cdot5^2\cdot7^2\cdot13^2\cdot41^2\cdot73^2\cdot6481^2, \\
(3^{48}-1)^2&=2^{12}\cdot5^2\cdot7^2\cdot13^2\cdot17^2\cdot41^2\cdot73^2\cdot97^2\cdot193^2\cdot577^2\cdot769^2\cdot6481^2.
\end{align}
Hence, for $1\leq n\leq 4$, we have 
\[
v_2(\det(3\bm{I}_{3\cdot 2^n}-\bm{U}_n))=v_2((2^{3\cdot 2^n}-1)^2)= 0\cdot 2^n+2\cdot n+4.
\]
In fact, by Remark \ref{rem:quantumwalk}, we have $v_2(\det(3\bm{I}_{3\cdot 2^n}-\bm{U}_n))=0\cdot 2^n+2\cdot n+4$ for all $n\geq 1$.

If $a=1/2$, then $v_2\left(\left(\frac{1}{2}\right)^{3}-1 \right)=-3\leq -1$. In this case, for every $n\geq 0$, we have
\[
v_2\left(\det\left(\frac{1}{2}\bm{I}_{3\cdot 2^n}-\bm{U}_n\right)\right)=v_2\left(\left(\left(\frac{1}{2}\right)^{3\cdot 2^n}-1 \right)^2\right)=-6\cdot 2^n+0\cdot n+0.
\]
\end{ex}
\bibliographystyle{amsalpha}
\bibliography{AMT24.bib}
\end{document}